\documentclass{article}
\usepackage{hyperref}
\usepackage{a4wide}
\usepackage{amsmath}
\usepackage{amssymb}
\usepackage{amsthm}
\usepackage{cleveref}
\usepackage{graphicx}
\usepackage{subcaption}
\usepackage{lipsum}
\usepackage{tabularx}
\usepackage{tikz}
\usepackage{tkz-graph}
\usetikzlibrary{shapes,arrows}
\usetikzlibrary{positioning,calc,intersections}
\newtheorem{theorem}{Theorem}
\newtheorem{remark}{Remark}

\newcommand{\R}{\mathbb{R}}

\renewcommand{\d}{{\mathrm{d}}}
\newcommand{\cof}{{\mathrm{cof}}}
\newcommand{\oneD}{_{\text{1D}}}
\newcommand{\twoD}{_{\text{2D}}}
\newcommand{\threeD}{_{\text{3D}}}
\newcommand{\domainVol}{{\Omega\threeD}}
\newcommand{\domainPln}{{\Omega\twoD}}
\newcommand{\imgVol}{u\threeD}
\newcommand{\imgPln}{u\twoD}
\newcommand{\done}{d_1}
\newcommand{\dtwo}{d_2}
\newcommand{\F}{{\mathcal F}}
\newcommand{\Reg}{{\mathcal R}}
\newcommand{\J}{\mathcal{J}}
\newcommand{\E}{\mathcal{E}}
\newcommand{\D}{\mathrm{D}}
\newcommand{\TVol}{\mathcal{T}_{\mathrm{3D}}}
\newcommand{\TPln}{\mathcal{T}_{\mathrm{2D}}}
\newcommand{\id}{{\mathrm{id}}}
\newcommand{\vol}{{\mathrm{vol}}}
\newcommand{\diam}{{\mathrm{diam}}}
\newcommand{\tr}{\mathrm{tr}}
\newcommand{\loc}{{\mathrm L}}
\newcommand{\nonloc}{{\mathrm{NL}}}

\newcommand{\todo}[1]{\textcolor{red}{[{\bf todo:} #1]}}

\newcommand{\notinclude}[1]{}

\title{Elastic 3D-2D Image Registration}
\author{Paul Striewski\footnote{Applied Mathematics M\"unster, \nolinkurl{Paul.Striewski@uni-muenster.de}} \and Benedikt Wirth\footnote{Applied Mathematics M\"unster, \nolinkurl{Benedikt.Wirth@uni-muenster.de}}}
\date{}

\begin{document}
\maketitle 
\begin{abstract}
We propose a method to non-rigidly align a three-dimensional (3D) volumetric image with a two-dimensional (2D) planar image representing a projection of the 
deformed volume.
The application in mind comes from biological studies in which 2D intravital microscopy videos of living tissue are recorded,
after which the tissue is excised and a more detailed 3D volume microscopy is performed.
Coregistration of both data sets allows to combine the temporal (but 2D) information with more detailed spatial 3D information.

Our approach is variational and uses a hyperelastic deformation regularization, as is appropriate for biological material.
As a particular feature, the out of plane deformation is estimated based on the out of focus blur inside the 2D microscopy image.
The approach becomes computationally feasible through the use of a coarse-to-fine optimization strategy and higher order optimization methods.
\end{abstract}

\section{Introduction}\label{Sct:Intro}
In this article we propose an elastic 3D-2D image registration technique, which is motivated by a challenging image processing problem arising in biological
studies of leucocyte extravasation into inflamed tissue in mice. A common experimental procedure consists of irritating a certain tissue region and performing 
two-dimensional intravital microscopy (IVM)
to monitor the reactions on the cellular level, see for instance \cite{Buscher2016}. After a sufficiently long video 
sequence has been recorded, the animal is sacrificed, the 
tissue is removed, stained, and fixed, and a single, volumetric 3D image is acquired using a 3D confocal microscope. The two obtained datasets therefore show 
essentially the same tissue region, one being an elastically deformed configuration of the other.
While the IVM recording shows in addition the temporal behaviour of the moving leucocytes, the confocal microscopy image contains more detailed structures such 
as the blood vessel basement membranes.
Aligning and merging both datasets allows to combine the enhanced structural information obtained by the 3D confocal microscope with temporal information 
gained from the IVM.\\

Both the 2D and 3D dataset use flourescent staining and thus possess multiple colour channels corresponding to different flourescent dyes.
In particular, one channel of both images contains the same biological structures (for instance the endothelial cell membranes) so that those can be used for 
registration.
Our approach for registering the images is based on mimicking the image acquisition system of the 2D microscope,
including a non-rigid tissue deformation and a blurring and projection step.
The intensities of the blurred and projected image volume are then compared with those of a 2D reference image (for instance a single frame of the IVM video), 
see \cref{fig:pipeline} for the overall procedure.
The goal of the registration process is then to find a deformation transforming the volumetric image in such a way that the similarity of the 
two datasets is maximal.
The set of admissible deformations is chosen in a way that reflects physical properties of the underlying tissue.\\

\begin{figure}
\center
\includegraphics[width=\textwidth]{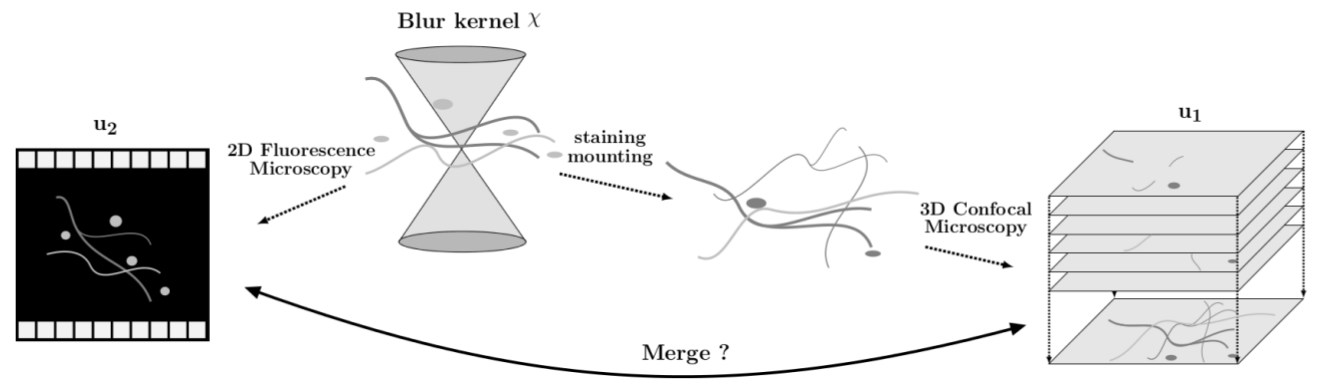}
\caption{Forward model of the relation between the 2D and 3D datasets.
Given a specimen of the tissue, 3D confocal microscopy yields a corresponding 3D image of voxels.
Taking the tissue configuration during confocal microscopy as a reference,
the 2D microscopy image is obtained by a (geometrically nonlinear) deformation and a subsequent 2D image acquisition
in which structures outside the focus plane are blurred during the projection onto the 2D image plane.}
\label{fig:pipeline}
\end{figure}

In this article, we present a complete mathematical description of the aforementioned procedure and, using the direct method in the calculus of variations, 
guarantee the existence of an admissible deformation which maximizes the similarity between the datasets. Although rigid or parameterized 3D-2D registration 
\cite{Ref:Otake,Ref:ZhengYu} and elastic 3D-3D registration \cite{Ref:Burger} has been studied before, elastic 3D-2D registration has so far not been attempted to the best of our knowledge.

Xu and Wan presented an intensity based 3D-2D registration technique for the matching of a CT volume to 2D X-Ray images and a description of how their method 
can be implemented on GPUs \cite{Ref:XuWan}. Studies on CT to X-Ray Fluoroscopy registration for guidance of surgical intervention were published by Otake et 
al. \cite{Ref:Otake} and Uneri et al.\ \cite{Ref:UneriEtAl}. All three articles employ rigid motions to obtain optimal alignment, which can be expected to 
yield satisfactory registration results for this specific type of problem. 

Also descriptions of nonrigid methods can be found in the literature. Zheng and Yu presented a spline based 3D-2D registration technique \cite{Ref:ZhengYu} for 
matching 3D CT data to 2D X-Ray images using a statistical deformation model. The article also references a number of publications which describe 
similar methods.
In contrast to our setting, in all above studies the 2D images result from a 3D volume via the X-Ray transform,
while our 2D images stem from optical microscopy and thus depict structures outside the focus plane with blur.
On the one hand this blur makes the lateral alignment of structures between the 3D and 2D image more difficult (also computationally),
on the other hand it can provide additional information about the deformation along the viewing direction,
which in the above studies must come solely from the deformation model.

Heldmann and Papenberg presented a variational method to register 3D CT data to a sequence of 2D ultrasound slices \cite{Ref:HeldmannPapenberg}.
Note that their 2D images thus stem from point- or slicewise evaluation of a 3D volume rather than from an integral transform as in our and the above 
approaches.
The well-posedness of this problem thus requires high regularity of the 3D image to be registered and the corresponding deformation, which the authors provide 
employing 
curvature regularization. They include results of the application of their technique to clinical data, however, no in depth analysis of the presented 
energy functional is presented.
Berkels et al.\ studied 2D-3D surface registration, where graph representations of cortical surfaces had to be matched to 2D brain images. Unlike in our case, 
an optimal deformation $\psi: \R^2 \to \R^3$ had to be found. For the regularization of this highly ill-posed problem, the authors suggested a 
regularizer based on the second order thin plate spline energy of $\psi$ \cite{Ref:BerkelsSurfaceRegistration}. A different, yet still related, situation can 
arise in surface to surface matching problems, if surface regions are locally described by two dimensional images which then in turn can be processed further 
by 2D-2D image matching algorithms. Such methods were for example described by Merelli et al.\ \cite{Ref:MerelliEtAl2008,Ref:MerelliEtAl2011}.

Our method will estimate displacement in viewing direction from the severity of out of focus blur.
A closely related problem, known as \emph{depth-from-defocus}, is to estimate the distance between an object and the image acquisition device
from a whole image stack (rather than just one image as in our case) that was generated by capturing the object under varying focal settings.
The forward model described by Persch et al.\ \cite{Persch2017}, which mimics the image acquisition of a thin lens camera, is comparable to our 2D microscope model.
Likewise related is the article by Aguet et al.\ \cite{AguetVille2008}, which suggests a method of how an image stack,
produced by moving a sample through the different focal planes of a 2D microscope, can be combined into a single feature enriched image.

For further methods we refer to the extensive overview \cite{Markelj2012} of 3D-2D registration techniques with a focus on applications in image guided surgery.

In the next section we present the mathematical model and the analysis of existence of solutions in detail.
\Cref{sec:numerics} describes the numerical implementation, while \cref{sec:results} shows the behaviour of the method in carefully chosen test cases and a real biological dataset.

\section{Mathematical Formulation}
\label{Sec:MathematicalFormulation}
We continue with the mathematical description of our registration problem.
We first briefly describe the forward operator of the 2D microscope and subsequently a physically reasonable model of tissue deformations.

For simplicity, our exposition will consider the 3D and 2D unit domains $\domainVol = [-1,1]^3$ and $\domainPln = [-1,1]^2$, respectively.
The volumetric image stemming from 3D confocal microscopy is denoted $\imgVol:\domainVol \to [0,\infty)$ and is considered to represent the tissue reference 
configuration.
In comparison to the IVM the quality of the 3D microscopy process is typically good enough to assume $\imgVol$ noise- and blurfree.
The 2D microscopy image $\imgPln:\domainPln\to[0,\infty)$ will be interpreted as resulting from a non-rigid deformation of $\imgVol$ and a subsequent 
projection into the plane.
Above, both $\imgVol$ and $\imgPln$ are assumed to represent only that colour channel in which the same biological structures are visible;
furthermore we will assume both images to be uniformly bounded
(which is appropriate since the image intensities can be interpreted as the local concentration of fluorescent molecules).

The forward operator of 2D microscopy can be modelled by
\begin{equation*}
\F: L^\infty(\domainVol) \to L^\infty(\domainPln)\,,\quad
(\F u)(x_1,x_2) = [\chi \ast u](x_1,x_2, 0)\,,
\end{equation*}
where the blurring kernel $\chi \in L^1(\R^3)$ has compact support and encodes the microscope's point spread function.
Above, $L^p$ denotes the standard Lebesgue function space, and $\chi\ast u$ denotes convolution (after extending $u$ by zero outside $\domainVol$).
Due to
\begin{equation*}
|[\chi\ast u](x)-[\chi\ast u](\tilde x)|
=\left|\int_{\R^3}(\chi(x-z)-\chi(\tilde x-z))u(z)\,\d z\right|
\leq\|\chi(x-\cdot)-\chi(\tilde x-\cdot)\|_{L^1}\|u\|_{L^\infty}
\to_{\tilde x\to x}0
\end{equation*}
we have $\chi\ast u\in C_c^0(\R^3)$ (the space of continuous functions with compact support)
so that $\F$ is well-defined and even maps into $C^0(\domainPln)$ (the space of continuous functions on $\domainPln$).
The value $\chi((x_1,x_2,0)-z)$ can be interpreted as the amount of light recorded at $(x_1,x_2)$ by the 2D microscope from a point source of light at $z$,
which is readily seen from $(\F u)(x_1,x_2)=\int_{\R^3}\chi((x_1,x_2,0)-z)u(z)\,\d z$.
In our model we assume the blurring kernel to be spatially independent so that the observed blur of a point $z\in\domainVol$ only depends on its height $z_3$.
A simple example (used in our calculations due to lack of a properly measured point spread function) is given by
\begin{equation*}
\chi(x)=\begin{cases}
\frac1{\pi c^2x_3^2}&\text{if }|x_3|\leq1,x_1^2+x_2^2\leq c^2x_3^2,\\
0&\text{else,}
\end{cases}
\end{equation*}
which blurs each point $z$ outside the focus plane $\R^2\times\{0\}$ uniformly onto a disc of radius $cz_3$.

\begin{remark}[Compactness of forward operator]
The previous calculations even show that $\F$ is a compact operator from $L^\infty(\domainVol)$ to $C^0(\domainPln)$
(as expected for convolutions, even though we only extract a slice from the convolution result).
Indeed, due to $|[\chi\ast u](x)-[\chi\ast u](\tilde x)|\leq\|\chi(x-\cdot)-\chi(\tilde x-\cdot)\|_{L^1}\|u\|_{L^\infty}$,
the image under $\F$ of a bounded subset $A\subset L^\infty(\domainVol)$ is equicontinuous and thus compact in $C^0(\domainPln)$ by the Arzel\`a--Ascoli 
theorem.
However, in our analysis we will not make use of this fact.
\end{remark}

\begin{remark}[Less image regularity]
With additional conditions on the blurring kernel $\chi$ one may reduce the image regularity to $\imgVol\in L^q(\domainVol)$
and obtain a continuous linear forward operator $\F:L^q(\domainVol)\to L^q(\domainPln)$ for any $q\geq1$.
Indeed, assume that $\|\chi(\cdot,\cdot,x_3)\|_{L^1}\leq C$ uniformly for almost all $x_3\in\R$ and some $C>0$ (note that this is satisfied by our example 
kernel),
then for $u\in L^q(\domainVol)$ we have
\begin{equation*}
\F u(x_1,x_2)
=\int_\R u(\cdot,\cdot,x_3)*\chi(\cdot,\cdot,-x_3)\,\d x_3
\end{equation*}
using Fubini's theorem and thus $|\F u(x_1,x_2)|^q\leq K\int_\R|u(\cdot,\cdot,x_3)*\chi(\cdot,\cdot,-x_3)|^q\,\d x_3$ by Jensen's inequality for some $K>0$ 
depending on the support of $\chi$.
Therefore, by Fubini's theorem and Young's convolution inequality we have
\begin{multline*}
\|\F u\|_{L^q}^q
=\int_\domainPln|\F u|^q(x_1,x_2)\,\d(x_1,x_2)
\leq K\int_\R\int_\domainPln|u(\cdot,\cdot,x_3)*\chi(\cdot,\cdot,-x_3)|^q\,\d(x_1,x_2)\,\d x_3\\
\leq K\int_\R\|u(\cdot,\cdot,x_3)\|_{L^q}^q\|\chi(\cdot,\cdot,-x_3)\|_{L^1}^q\,\d x_3
\leq KC^q\int_\R\|u(\cdot,\cdot,x_3)\|_{L^q}^q\,\d x_3
\leq KC^q\|u\|_{L^q}^q\,.
\end{multline*}
\end{remark}

Between the 2D and the 3D image acquisition the tissue is deformed by a deformation $y:\domainVol\to\R^3$,
where for any point $x\in\domainVol$ the value $y(x)$ shall be interpreted as the new position assumed during the 3D image acquisition.
The inverse deformation $y^{-1}$ thus moves the reference configuration back into the configuration during the 2D microscopy
so that we expect $\imgPln=\F(\imgVol\circ y)$.
However, due to noise in the image acquisition and additional artefacts not included in our model (such as diffuse background signals) one cannot expect 
equality.
Instead we shall seek a deformation $y$ that leads to a small \emph{dissimilarity measure}
\begin{equation*}
\J^d[y] = \int_{\domainPln} d\left(\F(\imgVol\circ y)(x),\, \imgPln(x)\right) \, \d x\,,
\end{equation*}
where $d:\R\times\R \to [0,\infty)$ measures the distance between two image intensities
(again $\imgVol$ is extended by zero outside $\domainVol$).
The optimal choice of $d$ is in general determined by the type of noise contained in the image data.
We will focus on the $L^1$ and $L^2$ distance measures obtained with
\begin{equation*}
 \done(a,b) = |a-b|, \quad
 \dtwo(a,b) = |a-b|^2,
\end{equation*}
respectively.
The former is well-known to be appropriate if the data contains strong outliers, while the latter is appropriate for additive Gaussian noise.
Other choices include correlation-based distance measures (see for instance \cite[Sct.\,6.1]{Ref:BookModersitzki}) or the Kullback--Leibler Divergence 
$d_{\text{KL}}(a,b)=b-a+a\log\frac ab$ in case of Poisson noise.\\

The minimization of $\J^d[y]$ for general mappings $y:\domainVol\to\R^3$ is neither physically reasonable nor well-posed.
Indeed, $y$ might be discontinuous or not even measurable so that the composition $\imgVol\circ y$ does not make sense.
Thus we have to regularize the deformation $y$ by imposing additional constraints and adding an extra energy term.
Concerning physical constraints, $y$ should neither reverse the orientation of the deformed body $\domainVol$ nor should it introduce self-intersections of the 
material.
The first condition translates to
$$\det \nabla y > 0\text{ almost everywhere,}$$
whereas the second condition is imposed by additionally demanding
\begin{equation*}\label{eqn:InjectivityRequirement}
  \int_{\domainVol} \det\, \nabla y(x)\, \d x \ \leq \vol(y(\domainVol))\,,
\end{equation*}
where $\vol$ shall denote the three-dimensional Lebesgue measure.
The advantage of this constraint is that it guarantees the injectivity of $y$ if combined with $\det \nabla y > 0$ and that it also holds for the weak limit of 
a sequence $(y_k)_k$ satisfying the constraint, see for example \cite[Thm.\,7.9-1]{Ref:CiarletElasticity}.
We shall also constrain the maximal possible displacement according to
\begin{equation*}
\|y\|_{L^\infty}\leq\diam(\domainPln)
\end{equation*}
in order not to shift the three-dimensional volume out of the visible area.
In addition to those hard constraints, unreasonable growth or shrinkage of lengths, areas, or volumes by the deformation should be penalized.
Modelling the biological tissue as an elastic material, such a penalization can be achieved by adding an elastic deformation energy
\begin{equation*}
\Reg[y]=\int_\domainVol W(\nabla y)\,\d x
\end{equation*}
as regularization,
which will prevent physically unreasonable deformations with too high elastic energy.
Here, $W:\R^{3\times3}\to[0,\infty)$ represents the \emph{stored energy function}.
Since biological tissue is very soft, it typically undergoes considerable nonlinear deformation so that the elastic model should be geometrically nonlinear 
and, in particular, rigid motion invariant.
Furthermore, due to lack of better information we may assume a homogeneous and isotropic elastic constitutive law.
Together, the above assumptions imply that the stored energy function can be written as a function of the singular values or of the invariants of its argument 
(see for example \cite[Ch. 4]{Ref:CiarletElasticity}) so that
\begin{equation*}
W(A)=\hat W(\|A\|_F,\|\cof A\|_F,\det A)
\end{equation*}
with the Frobenius norm $\|A\|_F = \sqrt{\tr(A^T A)}$ and the cofactor matrix $\cof A=\det A A^{-T}$.
Note that $\|A\|_F$, $\|\cof A\|_F$, and $\det A$ control length, area, and volume changes respectively.
Furthermore, biological tissue is almost incompressible so that deviation from $\det\nabla y=1$ should be strongly penalized.
In our numerical experiments we will use the particular example
\begin{align*}
W(\nabla y) &= c_1 \|\nabla y\|_F^2
+ g(\det \nabla y)
\end{align*}
for a positive constant $c_1$  and a function $g:\R \to [0,\infty)$ such that $g(x) \to \infty$ for $x \to 0$. To favour deformations obeying $\det 
\nabla y \approx 1$, a term like $|\det\nabla y-1|^2$ can be incorporated into
the definition of $g$. Note that this function is polyconvex (that is, can be written as a convex function of $(\nabla y,\cof\nabla y,\det\nabla y)$),
which ensures weak lower semi-continuity of $\Reg$ as will be needed for the existence analysis.
Summarizing, we arrive at the optimization problem
$$\E^d[y]=\J^d[y]+\Reg[y]\to\min!$$
whose solution is the sought matching deformation.

\begin{remark}[Bayesian perspective]
The task of registering the two datasets can also be viewed as an inverse problem,
where a measurement $\imgPln$ is given alongside with a model of the forward operator $y\mapsto\F(\imgVol\circ y)$, which maps a deformation of the tissue into the 
corresponding projection onto a planar image.
Due to measurement noise and unknown modelling errors in the forward operator, the measurement is a random variable.
Likewise, the sought deformation $y$ can be interpreted as a random variable with different outcomes in different repetitions of the measurement.
From a Bayesian perspective, one would now like to maximize the conditional probability
\begin{align*}
P(\F(\imgVol\circ y)|\imgPln) = \frac{P(\imgPln|\F(\imgVol\circ y))P(y)}{P(\imgPln)}
\end{align*}
(we formally use $P$ like a probability density over an infinite-dimensional space).
After taking the negative logarithm, the optimization problem thus turns into
\begin{equation*}\label{Eq:LogarithmicMin}
  \min_y \ -\log P(\F(\imgVol\circ y)|\imgPln) = \min_y \left[ -\log P(\imgPln|\F(\imgVol\circ y)) - \log P(y) + \log P(\imgPln) \ \right]\,.
\end{equation*}
The term $\log P(\imgPln)$ is independent of $y$ and can be neglected,
and for the probability distribution of deformations it is reasonable to assume a Boltzmann-type distribution $P(y)\sim\exp(-\Reg[y])$, in which the 
probability decreases exponentially with increasing deformation energy.
Likewise, the probability distribution of the measurement $(\imgPln)_k$ in the $k$\textsuperscript{th} image pixel is taken as 
$P((\imgPln)_k|(v)_k)\sim\exp(-d(v_k,(\imgPln)_k))$, where $v_k=(\F(\imgVol\circ y))_k$ is the expected value in pixel $k$.
For instance, if the noise distribution is Gaussian we have $P((\imgPln)_k|(v)_k)\sim\exp(-d_2(v_k,(\imgPln)_k))$.
Integrating over all pixels we obtain $P(\imgPln|\F(\imgVol\circ y))\sim\exp(-\int_\domainPln d(\F(\imgVol\circ y)(x),\imgPln(x))\,\d x)=\exp(-\J^d[y])$.
Summarizing, we arrive at the same optimization problem, $\min_y\J^d[y]+\Reg[y]$.
\end{remark}


\begin{theorem}[Existence of minimizer]
\label{Thm:ExistenceMinimizer}
Let $\imgVol \in L^\infty(\domainVol)$ and $\imgPln \in L^\infty(\domainPln)$.
Also, let $\E^d = \J^{d} + \Reg$ for a finite nonnegative dissimilarity $d(\cdot,\cdot)$, convex and lower semi-continuous in its first argument, and a 
polyconvex lower semi-continuous stored energy function $W$ satisfying
\begin{equation*}
W(A) \geq C \left(\, \|A\|_F^p + \max\{0,\det A\}^{-r} \, \right) - \beta
\end{equation*}
for some exponents $p>3$, $r>0$ and constants $C,\beta>0$.
Then $\E^d$ admits a minimizing deformation $y$ on the set of admissible deformations
\begin{multline*}
\mathcal{A} = \left\{\vphantom{\int_{\domainVol}} y\in W^{1,p}(\domainVol, \R^3)\, \middle|\,\|y\|_{L^\infty}\leq\diam(\domainPln),\right.\\
\left.\det \nabla y > 0\text{ almost everywhere, }\int_{\domainVol} \det\nabla y(x)\, \d x \leq \vol(y(\domainVol)) \right\}
\end{multline*}
($W^{1,p}$ denotes the standard Sobolev space).
Furthermore, $y$ is almost everywhere injective.
\end{theorem}
\begin{proof}
We follow the direct method of the calculus of variations.
First note that $\E^d$ is bounded below by $-\beta\vol(\domainVol)$ and that $\inf_{\mathcal A}\E^d$ is finite due to $\E^d[\id]<\infty$.

\emph{Compactness: }
Consider a minimizing sequence $y_k\in\mathcal A$, $k=1,2,\ldots$, with $\E^d[y_k]\to\inf_{\mathcal A}\E^d$ monotonically as $k\to\infty$.
Due to the growth condition on $W$ we have $$\E^d[y_1]\geq\E^d[y_k]\geq C\|\nabla y_k\|^p-\beta\vol(\domainVol)$$
so that $\|\nabla y_k\|_{L^p}$ is uniformly bounded.
Together with the admissibility condition $\|y_k\|_{L^\infty}\leq\diam(\domainPln)$ we obtain uniform boundedness of $\|y_k\|_{W^{1,p}}$
so that we can extract a weakly converging subsequence (still indexed by $k$) $y_k\rightharpoonup y$ in $W^{1,p}(\domainVol;\R^3)$.
Due to $p>3$, by Sobolev embedding we may even assume $y_k\to y$ strongly in the space $C^{0,\alpha}(\domainVol;\R^3)$ of H\"older continuous functions with 
exponent $\alpha<1-p/3$.

\emph{Lower semi-continuity of $\Reg$: }
We have $\Reg[y]\leq\liminf_{k\to\infty}\Reg[y_k]$ by the properties of $W$.
Indeed, by H\"older's inequality $\cof\nabla y_k$ and $\det\nabla y_k$ are uniformly bounded in $L^{2p/3}(\domainVol)$ and $L^{p/3}(\domainVol)$, respectively, 
and thus converge for a subsequence.
By \cite[Thm. 7.6-1]{Ref:CiarletElasticity} we even have
\begin{equation*}
\cof\nabla y_k\rightharpoonup\cof\nabla y\text{ in }L^{2p/3}(\domainVol)
\quad\text{and}\quad
\det\nabla y_k\rightharpoonup\det\nabla y\text{ in }L^{2p/3}(\domainVol)\,.
\end{equation*}
Now Mazur's lemma implies the existence of a sequence of strongly and pointwise almost everywhere converging convex combinations
\begin{equation*}
\sum_{i=k}^{N_n}a_i^k(\nabla y_i,\cof\nabla y_i,\det\nabla y_i)\to(\nabla y,\cof\nabla y,\det\nabla y)
\end{equation*}
as $k\to\infty$, where $N_n\geq k$ and  the nonnegative coefficients $a_i^k$ sum up to one.
Since $W$ is polyconvex we can write $W(A)=\tilde W(\nabla A,\cof\nabla A,\det\nabla A)$ for a convex function $\tilde 
W:\R^{3\times3}\times\R^{3\times3}\times\R\to\R$.
Thus, with Fatou's lemma and the lower semi-continuity of $W$ we now have
\begin{multline*}
\Reg[y]
=\int_\domainVol\tilde W(\nabla y,\cof\nabla y,\det\nabla y)\,\d x
=\int_\domainVol\tilde W\left(\lim_{k\to\infty}\sum_{i=k}^{N_n}a_i^k(\nabla y_i,\cof\nabla y_i,\det\nabla y_i)\right)\,\d x\\
\leq\int_\domainVol\liminf_{k\to\infty}\sum_{i=k}^{N_n}a_i^k\tilde W(\nabla y_i,\cof\nabla y_i,\det\nabla y_i)\,\d x
\leq\liminf_{k\to\infty}\sum_{i=k}^{N_n}a_i^k\int_\domainVol\tilde W(\nabla y_i,\cof\nabla y_i,\det\nabla y_i)\,\d x\\
=\liminf_{k\to\infty}\sum_{i=k}^{N_n}a_i^k\Reg[y_i]
\geq\liminf_{k\to\infty}\Reg[y_k]\,.
\end{multline*}

\emph{Properties of limit function: }
The limit function $y$ lies in $\mathcal A$.
Indeed, by the uniform convergence $y_k\to y$ we have $\|y\|_{L^\infty}=\lim_{k\to\infty}\|y_k\|_{L^\infty}\leq\diam(\domainPln)$.
To see that $\det \nabla y > 0$ holds almost everywhere, consider the set
\begin{align*}
  S_{\varepsilon} = \{ x \in \domainVol \, |\, \det \nabla y(x) < \varepsilon \}
\end{align*}
for $\varepsilon>0$. The growth condition on $W$ and the lower semi-continuity of $\Reg$ imply
\begin{equation*}
\vol(S_{\varepsilon})(C\varepsilon^{-r}-\beta)
\leq \Reg[y]
\leq \liminf_{k\to\infty}\Reg[y_k]
\leq \liminf_{k\to\infty}\E^d[y_k]
\leq\E^d[y_1]\,.
\end{equation*}
Thus, $\vol(S_{\varepsilon}) \to 0$ as $\varepsilon \to 0$ implies $\det \nabla y > 0$ almost everywhere.
Likewise, due to the weak convergence $\det\nabla y_k\rightharpoonup\det\nabla y$ and the convergence $y_k\to y$ in $C^{0,\alpha}(\domainVol)$ we have
\begin{equation*}
\int_\domainVol\det\nabla y\,\d x
=\lim_{k\to\infty}\int_\domainVol\det\nabla y_k\,\d x
\leq\lim_{k\to\infty}\vol(y_k(\domainVol))
=\vol(y(\domainVol))\,.
\end{equation*}

Furthermore, $y$ is injective almost everywhere, that is, the cardinality
\begin{equation*}
N(y\,|\,v)=\mathrm{card}(y^{-1}(\{v\}))
\end{equation*}
equals $1$ for almost every $v\in y(\domainVol)$.
Indeed, by the change of variables formula for Sobolev functions \cite[Thm.\,2]{MaMi73} we have
\begin{equation*}
\vol(y(\domainVol))
\leq\int_{y(\domainVol)}N(y\,|\,v)\,\d v
=\int_\domainVol\det\nabla y(x)\,\d x\,.
\end{equation*}
Since also the opposite inequality holds, we must have equality and $N(y\,|\,v)=1$ almost everywhere.

\emph{Lower semi-continuity of $\J^d$: }
Note that we have $\imgVol\circ y_k\to\imgVol\circ y$ as $k\to\infty$ in any $L^q(\domainVol)$ with $q\in[1,\infty)$.
Indeed, for a Dirac sequence $G_{\delta}$ of smooth mollifiers we have
\begin{multline*}
\|\imgVol \circ y_k-\imgVol \circ y \|_{L^q}
\leq\|\imgVol \circ y_k-(G_{\delta} \ast \imgVol) \circ y_k\|_{L^q}\\
+ \|(G_{\delta} \ast \imgVol) \circ y_k-(G_{\delta} \ast \imgVol) \circ y\|_{L^q}
+ \|(G_{\delta} \ast \imgVol) \circ y-\imgVol \circ y\|_{L^q}\,.
\end{multline*}
Abbreviating $s=(r+1)/r$ and employing H\"older's inequality, for the first summand we obtain
\begin{align*}
\int_{\domainVol} &|\imgVol \circ y_k - (G_{\delta} \ast \imgVol) \circ y_k |^q \, \d x
\leq \int_{\domainVol}\left| \frac{| \imgVol \circ y_k- (G_{\delta} \ast \imgVol) \circ y_k |^q }{ \det \nabla y_k^{\frac{1}{s}} } \det \nabla 
y_k^{\frac{1}s}\right| \, \d x\\
&\leq \left( \int_{\domainVol} | \imgVol \circ y_k -(G_{\delta} \ast \imgVol) \circ y_k |^{qs} \det \nabla y_k \, \d x 
\right)^{\frac{1}s}\left(\int_\domainVol\det\nabla y_k^{-\frac1{s-1}}\,\d x\right)^{1-\frac1s}\\
&=\left( \int_{y_k(\domainVol)} | \imgVol -G_{\delta} \ast \imgVol |^{qs} N(y_k\,|\,v) \, \d v \right)^{\frac{1}s}\left(\int_\domainVol\det\nabla y_k^{-r}\,\d 
x\right)^{\frac1{r+1}}\\
&\leq \|\imgVol - G_{\delta} \ast \imgVol\|_{L^{qs}(\R^3)}^q \left(\frac{\Reg[y_k]+\beta\vol(\domainVol)}C\right)^{\frac1{r+1}}\,,
\end{align*}
where we used the change of variables for Sobolev functions \cite[Thm.\,2]{MaMi73} as well as $N(y_k\,|\,v)=1$ for almost all $v\in y_k(\domainVol)$ (by the 
same argument as for $y$).
Since $\imgVol\in L^{qs}(\domainVol)$ and $\Reg[y_k]\leq\E^d[y_k]\leq\E^d[y_1]$, the right-hand side converges to $0$ as $k\to\infty$ and then $\delta\to0$.
For the second summand we observe
\begin{equation*}
\|(G_{\delta} \ast \imgVol ) \circ y_k-(G_{\delta} \ast \imgVol ) \circ y \|_{L^q}^q
\leq L_{\delta}^q \|y_k- y\|_{L^q}^q\,, 
\end{equation*}
$L_{\delta}$ being the Lipschitz constant of $G_{\delta} \ast \imgVol$.
Again, letting first $k\to\infty$ and then $\delta\to0$ the right-hand side converges to $0$.
The third summand is treated like the first so that in summary $\imgVol\circ y_k\to\imgVol\circ y$.

Due to $\imgVol\in L^\infty(\domainVol)$, the composition $\imgVol\circ y_k$ is uniformly bounded in $L^\infty(\domainVol)$
so that any subsequence contains another weakly-* converging subsequence in $L^\infty(\domainVol)$.
Due to the strong convergence $\imgVol\circ y_k\to\imgVol\circ y$ in $L^q(\domainVol)$, the limit must be the same and thus
\begin{equation*}
\imgVol\circ y_k\stackrel*\rightharpoonup\imgVol\circ y\quad\text{in}\quad L^\infty(\domainVol)
\end{equation*}
for the whole sequence.
Furthermore it is straightforward to check that $\F$ is the adjoint operator to
\begin{equation*}
\F':L^1(\domainPln)\to L^1(\domainVol),\quad
g\mapsto (x\mapsto[g*\chi(-\cdot,-\cdot,-x_3)](x_1,x_2))\,,
\end{equation*}
which is a bounded linear operator due to
\begin{equation*}
\|\F'g\|_{L^1}
=\int_{-1}^1\|g*\chi(-\cdot,-\cdot,-x_3)\|_{L^1}\,\d x_3
\leq\int_{-1}^1\|g\|_{L^1}\|\chi(-\cdot,-\cdot,-x_3)\|_{L^1}\,\d x_3
=\|g\|_{L^1}\|\chi\|_{L^1}
\end{equation*}
by Young's convolution inequality.
As a consequence, $\F(\imgVol\circ y_k)\stackrel*\rightharpoonup\F(\imgVol\circ y)$ in $L^\infty(\domainPln)$, since for any $g\in L^1(\domainPln)$ we have
\begin{equation*}
\int_\domainPln g\F(\imgVol\circ y_k)\,\d x
=\int_\domainVol\F'(g)\,\imgVol\circ y_k\,\d x
\to\int_\domainVol\F'(g)\,\imgVol\circ y\,\d x
=\int_\domainPln g\F(\imgVol\circ y)\,\d x
\end{equation*}
as $k\to\infty$.
The convexity of $d$ in its first argument now implies $\liminf_{k\to\infty}\J^d[y_k]\geq\J^d[y]$, as desired.

Summarizing, $\E^d[y]=\J^d[y]+\Reg[y]\leq\liminf_{k\to\infty}\J^d[y_k]+\Reg[y_k]=\liminf_{k\to\infty}\E^d[y_k]=\inf_{\mathcal A}\E^d$ so that $y\in\mathcal A$ 
must be a minimizer.
\end{proof}


\section{Numerical Implementation}\label{sec:numerics}
In this section we discuss the discretization and numerical minimization of the energy functional
\begin{align*}
 \E^d[y] &= \int_{\domainPln}d\left(\,\F(\imgVol \circ y)(x),\, \imgPln(x)\, \right) \, \d x + \int_{\domainVol} W(\nabla y(x)) \ \d x\,,
\end{align*}
which is nontrivial due to the nonlocal convolution operator in $\F$, the composition of discretized functions, and the nondifferentiability of the discretized functions.
As before, $d$ denotes either the Euclidean distance $\done(x,y) = |x-y|$ or its square $d_2(x,y) = |x-y|^2$, but other choices can be implemented 
in the same way.
To obtain a differentiable functional in the former case (which will allow simpler numerics), we make the modification $\done(x,y) = \sqrt{(x-y)^2+\delta^2}$ 
with $\delta > 0$ a small regularization parameter. 
In our implementation the stored energy function $W$ has the form
\begin{align}
W(A) &= c_1 \|A\|_F^2 + c_2(\det A)^{-1} + c_3(1-\det A)^2 + D \label{Eq:StoredEnergyFunction}
\end{align}with constants $c_1,c_2,c_3 \geq 0$, $D\in \R$ such that $W\geq0$ and $W(S) = 0$ for rotation matrices $S$.
This specific choice violates the growth condition of \cref{Thm:ExistenceMinimizer}
(which was needed to apply a change of variables formula for Sobolev functions, while the lower semi-continuity of $\Reg$ could also be obtained for weaker 
growth conditions \cite[Thm.\,3.6]{Pedregal00}),
however, we observed no indication of degeneration of the deformations in our numerical experiments so that the above choice seemed sufficient.
Note that other stored energy functions can be implemented just as well, for instance the strain energy densities of Neo-Hookean materials,
\begin{align*}
 W(A) = \frac{1}{2}\mu(\|A\|_F^2 - 2 \log(\det A)) + \frac{\lambda}{2}(1-\det A)^2 - \frac{3\mu}{2}\, ,
\end{align*}which only differ from our choice by the slightly weaker penalty term for volume compression.
Note also that in our experiments the constraints $\int_{\domainVol} \det\, \nabla y(x)\, \d x \leq \vol(y(\domainVol))$ and 
$\|y\|_{L^\infty}\leq\diam(\domainPln)$ were always satisfied without explicit enforcement.

Assuming a twice differentiable 3D image $\imgVol$, the first and second G\^ateaux derivatives of $\J^d$ in $y\in 
\mathcal{A}$ for suitable variations $\phi$ and $\psi$ are
\begin{align*}
 \partial \J^d[y](\phi) &=  \int_{\domainPln} \partial_1d(\, \F(\imgVol\circ y)(x),\, \imgPln(x))\, \left(\chi \ast [\phi\cdot(\nabla \imgVol)\circ y] 
\right)(x_1,x_2,0)\ \d x\,,  \label{Eq:GateauxDerivatives} \\
 \partial^2 \J^d[y](\phi, \psi) &=\mathcal H^\loc(\phi,\psi)+\mathcal H^\nonloc(\phi,\psi)\quad\text{with}\\
\mathcal H^\loc(\phi,\psi)
&= \int_{\domainPln} \partial_1d(\, \F(\imgVol\circ y)(x),\, \imgPln(x))\, \left(\chi \ast \left[\psi^T\,
((\mathrm{Hess}\, \imgVol)\circ y) \, \phi\right]\right)(x, 0) \, \d x\,, \\
\mathcal H^\nonloc(\phi,\psi)
 &= \int_{\domainPln} \partial_1^2 d(\, \F(\imgVol\circ y)(x),\, \imgPln(x))\, \left(\chi \ast [\phi\cdot(\nabla \imgVol)\circ y]\right)(x_1,x_2,0)\,\left(\chi 
\ast [\psi\cdot(\nabla \imgVol)\circ y]\right)(x_1,x_2,0)\, \d x\, , \notag
\end{align*}
where $\mathrm{Hess}\, \imgVol$ denotes the Hessian of $\imgVol$.
While for convex $d$ the second summand in $\partial^2 \J^d[y]$ is always positive semi-definite, the first summand may destroy this definiteness.
For the sake of completeness, we also write down the expressions for the first and second G\^ateaux derivative of the hyperelastic regularizer.
Rewriting $W$ defined in \eqref{Eq:StoredEnergyFunction} in the form
\begin{equation*}
  W(A)=\bar W(\|A\|_F^2,\det A)
  \quad\text{with}\quad
  W(I_1, I_3) = c_1 I_1 + c_2 I_3^{-1} + c_3(1-I_3)^2 + D\, ,
\end{equation*}
its partial derivatives are given by
\begin{align*}
  \partial_1\bar W(I_1, I_3) &= c_1\,, & \partial_2\bar W(I_1, I_3) &= -c_2 I_3^{-2} - 2c_3 (1-I_3)\,, \\
  \partial_1^2\bar W(I_1, I_3) = \partial_1\partial_2\bar W(I_1, I_3) &= 0\,, & \partial_2^2\bar W(I_1, I_3) &= 2\, c_2 I_3^{-3} + 2c_3\, .
\end{align*}
With the help of the identities $\partial_A\det(A)(B) = \tr(B^T\cof A)$ and $\partial_A(\|A\|_F^2)(B) = 2\tr(B^TA)$ and abbreviating $I_1=\|\nabla y\|_F^2$ and $I_3=\det\nabla y$,
the G\^ateaux derivatives of the hyperelastic regularizer read
\begin{align*}
\partial \Reg[y](\phi) &= \int_{\domainVol} 2\partial_1\bar W(I_1, I_3)\, \tr(\nabla\phi^T\nabla y) + \partial_2\bar W(I_1, I_3)\, \tr(\nabla\phi^T 
\cof\nabla y)\, \d x\,, \label{Eq:GateauxDerivativesRegularizer} \\
\partial^2 \Reg[y](\phi, \psi) &= \int_{\domainVol} 2\partial_1\bar W(I_1, I_3)\tr(\nabla\phi^T\nabla\psi) + \partial_2^2\bar W(I_1, I_3)\, \tr(\nabla\psi^T \cof\nabla y)\, \tr(\nabla \phi^T \cof\nabla y) \notag \\
& \qquad\ + \frac{\partial_2\bar W(I_1, I_3)}{I_3}\, \left[ \tr(\nabla \psi^T \cof\nabla y)\,\tr(\nabla\phi^T\cof\nabla y) -\tr(\nabla \phi^T\cof\nabla y\,\nabla\psi^T\cof\nabla y) \right] \, \d x\,.
\notag
\end{align*}

\paragraph{Spatial discretization.}
The image domains $\domainPln$ and $\domainVol$ are discretized by two dyadically nested hierarchies of regular rectilinear grids $(\TPln^n)_{n=1,\ldots, N}$ and $(\TVol^n)_{n=1,\ldots, N}$, respectively.
The number of nodes in the $n$\textsuperscript{th} grid along each coordinate direction is $M_n=2^n+1$,
where the resolution of the given discrete input images determines $N$, the level of the 
finest grids $\TPln^N$ and $\TVol^N$. 

Introducing multilinear Finite Element basis functions on each grid gives rise to a hierarchy of $C^0$-Finite Element 
spaces $X^n = (X^n_{\twoD}\times X^n_{\threeD})_{n=1,\ldots,N}$ with
$X^n \subset X^m$ whenever $n \leq m$
and corresponding restriction and prolongation operators chosen as follows.
On a one-dimensional grid with $M_n$ nodes a Finite Element function $\xi$ can be identified with the vector $(\xi_1,\ldots,\xi_{M_n})$ of its nodal values.
On such a grid we define the one-dimensional restriction and prolongation operators as
\begin{align*}
\mathfrak R_n^{\oneD}: \left(\xi_j\right)_{1\leq j \leq M_n}\ &\mapsto\ 
\left(\tfrac{\xi_{2j-2}+2\xi_{2j-1}+\xi_{2j}}4\right)_{1\leq j\leq (M_n+1)/2}\,, \\
\mathfrak P_n^{\oneD}: \left(\xi_j\right)_{1\leq j \leq M_n}\ &\mapsto\ (\xi_1, \tfrac{\xi_1+\xi_2}2, 
\xi_2, \tfrac{\xi_2+\xi_3}2, \ldots, \xi_{M_n-1}, \tfrac{\xi_{M_n-1}+\xi_{M_n}}2, \xi_{M_n})
\end{align*}
(for ease of notation we set $\xi_0=\xi_1$ and $\xi_{M_n+1}=\xi_{M_n}$).
The restriction and prolongation operators 
\begin{equation*}
\mathfrak R_n: X^n \to X^{n-1}, \quad \mathfrak P_n: X^{n} \to X^{n+1}
\end{equation*}
are then obtained by applying consecutively $\mathfrak R_n^{\oneD}$ and $\mathfrak P_n^{\oneD}$ along each coordinate direction of the grid.

Fixing a grid $\TVol^n$ with $N_n=(2^n+1)^3$ nodes $(x^1, \ldots, x^{N_n})$, we denote the Finite Element basis functions on $\TVol^n$ by $\phi_1^n,\ldots,\phi_{N_n}^n$. The Finite Element 
representation of $\imgVol$ in $X^N_{\threeD}$ is taken as $\imgVol^N=\sum_{j=1}^{N_N} \imgVol(x^j) \phi_j^N$ and in $X^n_{\threeD}$ as $\imgVol^n=\mathfrak R_{n+1}\cdots\mathfrak R_N\imgVol^N$ 
(note that we will denote discretized functions on grid level $n$ with a superscript $n$). Similarly, the discretized deformation is expressed as $y^n = \sum_{j=1}^{N_n} y^n_j \phi_j^n$ with nodal coefficients $y^n_j\in\R^3$.
The discretized version $\imgPln^n$ of $\imgPln$ on grid $\TPln^n$ is defined in an analogous way.

\paragraph{Discretized cost functional.}
The evaluation of the data term $\J^d$ and its derivatives requires the computation of $\F(\imgVol \circ y)$. On the discretized level, the convolution contained in $\F$ is computed with the help of the discrete Fourier transform (DFT). To this end, $\imgVol^n \circ y^n$ and the discretized convolution kernel $\chi^n$ are evaluated at all grid points, and the resulting grid functions are padded with zeros so as to emulate the DFT using the fast Fourier transform on a periodic grid (in our case using routines of the FFTW project \url{http://www.fftw.org/}). The resulting grid function specifies the nodal values of a multilinear Finite Element function, which is then restricted to the $x_1$-$x_2$-plane to yield the discretized forward operator $F^n(\imgVol^n \circ y^n)$.
Using second order Gaussian quadrature on each element, the discrete analogue of $\J^d$ is now computed as
\begin{align*}
  J^d_n[y^n] = A_n\sum_{q} w\twoD^q\, d( F^n(\imgVol^n \circ y^n)(q),\, \imgPln^n(q)) \,,
\end{align*}
where $A_n$ stands for the area of each element in $\TPln^n$, $w\twoD^q$ denotes the quadrature weights, and the sum is taken over all quadrature points $q$.
Similarly, the discretized regularizer is evaluated via
\begin{equation*}
R_n[y^n]=V_n\sum_qw\threeD^q\,W(\nabla y^n(q))
\end{equation*}
with $V_n$ the volume of each element in $\TVol^n$. Finally, $E^d_n=J^d_n+R_n$.

\paragraph{Discretized functional derivatives.}
For the numerical evaluation of $\partial\J^d$ we exploit that the adjoint operator to a convolution is the cross-correlation.
In more detail, for functions $f:\domainPln\to\R$, $g:\domainVol\to\R$ we have
\begin{multline*}
\int_\domainPln f(x)[\chi\ast g](x_1,x_2,0)\,\d x
=\int_\domainPln\int_\domainVol f(x)\chi((x_1,x_2,0)-y)g(y)\,\d y\,\d x\\
=\int_\domainVol\int_\domainPln f(x)\chi((x_1,x_2,0)-y)\,\d x\,g(y)\,\d y
=\int_\domainVol[\chi\diamond f](y)g(y)\,\d y\,,
\end{multline*}
where we used Fubini's theorem and $[\chi\diamond f](y)=[\chi(-\cdot,-\cdot,-y_3)\ast f](y_1,y_2)$ denotes the adjoint operator to the convolution evaluated in the $x_1$-$x_2$-plane.
Denoting by $\ast^n$ our discrete approximation of the convolution described previously,
our discrete approximation of $\diamond$ applied to two Finite Element functions $f^n\in X\twoD^n$ and $\chi^n\in X\threeD^n$ is computed as the Finite Element function
\begin{equation*}
\chi^n\diamond^nf^n=\chi^n\ast^n Bf^n\,,
\end{equation*}
where $B:X\twoD^n\to X\threeD^n$ is the operator copying all nodal function values from $\TPln^n$ into the $x_1$-$x_2$-plane of $\TVol^n$ and leaving all other nodal values $0$.
Using the above notation we can write
\begin{equation*}
\partial\J^d[y](\phi)
=\int_\domainVol\left[\phi \cdot (\nabla \imgVol)\circ y\right](x)\,\left[\chi\diamond\partial_1d(\F(\imgVol\circ y,\imgPln)\right](x)\,\d x\,.
\end{equation*}
Correspondingly, using second order Gaussian quadrature, the discretized derivative is calculated for each Finite Element basis function $\phi_j^{n,k}=e_k\phi_j^n$ (with $e_k$ the $k$\textsuperscript{th} Cartesian unit vector) as
\begin{equation*}
\partial J^d_n[y^n](\phi_j^{n,k})=V_n\sum_qw\threeD^q\left[\phi_j^{n,k} \cdot (\nabla \imgVol^n)\circ y^n\right]\!(q)\,[\chi^n\diamond^n(\partial_1 d(F^n(\imgVol^n \circ y^n, \imgPln^n))](q)\,.
\end{equation*}
The discretized analogue of $\partial\Reg$ can be written as
\begin{align*}
  \partial R_n[y^n](\phi_j^{k,n}) &= V_n \sum_{q} w\threeD^q \tr\left((\nabla\phi_j^{k,n})^T(q)\left[2c_1 \nabla y^n(q) - \left[c_2 \det(\nabla y^n(q))^{-2} + 2c_3(1-\det(\nabla y^n(q)))  \right] 
\mathrm{cof}(\nabla y^n(q)) \right]\right) \, ,
\end{align*}
and $\partial E^d_n=\partial J^d_n+\partial R_n$.

\paragraph{Discretized second derivatives.}
The second derivative of $\J^d$ can be written as the sum $\partial^2\J^d[y]=\mathcal H^\loc+\mathcal H^\nonloc$ of a local and a nonlocal linear operator,
which are discretized separately.
Indeed, while $\mathcal H^\loc(\phi,\psi)$ is nonzero only if $\phi$ and $\psi$ have overlapping support,
the sparsity of a matrix representation of $\mathcal H^\nonloc(\phi,\psi)$ depends on the 
size of the blurring kernel $\chi$ and will typically be very low, making the storage of the 
assembled matrix impractical. However, during our numerical optimization we will only apply iterative solvers like BiCGStab, which only require the evaluation of matrix-vector products.
Due to the tensor product structure of the integrand of $\mathcal H^\nonloc$, the application of the linear operator can be implemented efficiently as follows.
Again exploiting the relation between convolution and the operator $\diamond$ we can rewrite
\begin{equation*}
\mathcal H^\nonloc(\phi,\psi)
=\int_\domainVol\left[\chi\diamond\left(\partial_1^2d(\F(\imgVol\circ y),\imgPln)(\chi\ast[\psi\cdot(\nabla\imgVol)\circ y])(\cdot,\cdot,0)\right)\right]\!(x)\,
[\phi\cdot(\nabla\imgVol)\circ y](x)\,\d x\,.
\end{equation*}
Thus, given a Finite Element function $\psi^n\in(X\threeD^n)^3$, for each Finite Element basis function $\phi_j^{k,n}=e_k\phi_j^n$ we can compute the discretized analogue
\begin{multline*}
H^\nonloc_n(\phi_j^{k,n},\psi^n)
=V_n\sum_qw\threeD^q\big[\chi^n\diamond^n\big(\partial_1^2d(F^n(\imgVol^n\circ y^n),\imgPln^n)\\
(\chi^n\ast^n[\psi^n\cdot(\nabla\imgVol^n)\circ y^n])(\cdot,\cdot,0)\big)\big](q)\,
[\phi_j^{k,n}\cdot(\nabla\imgVol^n)\circ y^n](q)\,.
\end{multline*}
The operator $\mathcal H^\loc$ is discretized as
\begin{equation*}
H^\loc_n(\phi_j^{k,n},\phi_i^{k,n})
=V_n\sum_qw\threeD^q\left[\phi_j^{k,n}\cdot\,((\mathrm{Hess}^n\, \imgVol^n)\circ y^n) \, \phi_i^{k,n}\right](q)\,
\left[\chi^n\diamond^n\partial_1 d(F^n(\imgVol^n \circ y^n, \imgPln^n)\right](q)\,,
\end{equation*}
where $\mathrm{Hess}^n\,\imgVol^n$ is defined weakly as in mixed Finite Elements approaches.
Indeed, as an artefact of our discretization, the piecewise multilinear Finite Element function $\imgVol^n$ does not possess a weak second derivative
(part of its distributional second derivative is concentrated on the element boundaries),
yet second order information $\mathrm{Hess}\imgVol$ is helpful for the registration and should not be neglected in the discretization.
Thus we define $\mathrm{Hess}^n\imgVol^n$ via
\begin{equation*}
\int_\domainVol\tr\left(\psi^T\mathrm{Hess}^n\imgVol^n\right)\,\d x
=\int_{\partial\domainVol}n^T\psi\nabla\imgVol^n\,\d x-\int_\domainVol\mathrm{div}\psi\cdot\nabla\imgVol^n\,\d x
\qquad\text{for all }\psi\in(X\threeD^n)^{3\times3}\,,
\end{equation*}
where $n$ denotes the unit outward normal to $\partial\domainVol$.
This amounts to solving a linear system of the form $MV=LU\threeD$ for the nodal value vector $V$ of $\mathrm{Hess}^n\imgVol^n$
with $M$ a mass matrix, $L$ a stiffness matrix, and $U\threeD$ the vector of nodal values of $\imgVol^n$.
Note that the matrix representation of $H^\loc_n$ is just a weighted mass matrix.

\paragraph{Numerical optimization.}
We tested and compared the performance of gradient-based methods, in particular nonlinear conjugate gradient and quasi-Newton methods, and a second order line search or trust region Newton method.
Below we provide a few details on the latter (the implementation of the former being straightforward).

The line search Newton method for the numerical minimization of $E^d_n$ over $(X\threeD^n)^3$ takes the form
\begin{align*}
  y_{k+1}^n &= y_k^n - \gamma_k(\partial^2 E^d_n[y_k^n])^{-1} \partial E^d_n[y_k^n]\,, \quad k \geq 0\, ,
\end{align*}
where we compute the step size $\gamma_k>0$ using a backtracking line search with Armijo's condition
and where the inverse linear operator is applied using BiCGStab (note that due to the lack of sparsity in $H^\nonloc_n$, the linear system has to be solved iteratively).
However, while $H^\nonloc_n$ is always positive semidefinite, $H^\loc_n$ and $\partial^2 R_n$ can be indefinite so that $\partial^2E^d_n$ may be so as well.
Consequently, the Newton step may not be a descent direction, and the iteration might converge to a saddle point.
To compensate a possible lack of positive definiteness we add a scalar multiple $\lambda_k$ of the identity to the Hessian operator,
\begin{align*}
  y_{k+1}^n &= y_k^n - \gamma_k(\partial^2 E^d_n[y_k^n] + \lambda_k \mathrm{id})^{-1} \partial E^d_n[y_k^n]\,,\quad k \geq 0\,, \label{ModifiedNewton}
\end{align*}
where $-\lambda_k$ should approximate the most negative eigenvalue of $H^\loc[y_k^n] + \partial^2 R_n[y_k^n]$. To determine $\lambda_k$, we use the procedure described in 
\cite[Sct.\,8.5.2, Thm.\,8.5.1]{Ref:GolubLoan}, which consists of a number of truncated Lanczos iterations to obtain a symmetric tridiagonal 
approximation $T\in\R^{m\times m}$ to $H^\loc[y_k^n] + \partial^2 R_n[y_k^n]$ and a subsequent computation of its characteristic polynomial,
whose smallest zero is found via a bisection method.

\begin{figure}
\def\angle{0}
\def\radius{3}
\def\cyclelist{{"orange","blue","red","green"}}
\newcount\cyclecount \cyclecount=-1
\newcount\ind \ind=-1
\begin{tabular}{cc}
\begin{subfigure}[t]{0.5\textwidth}\centering\includegraphics[width=0.9\columnwidth]{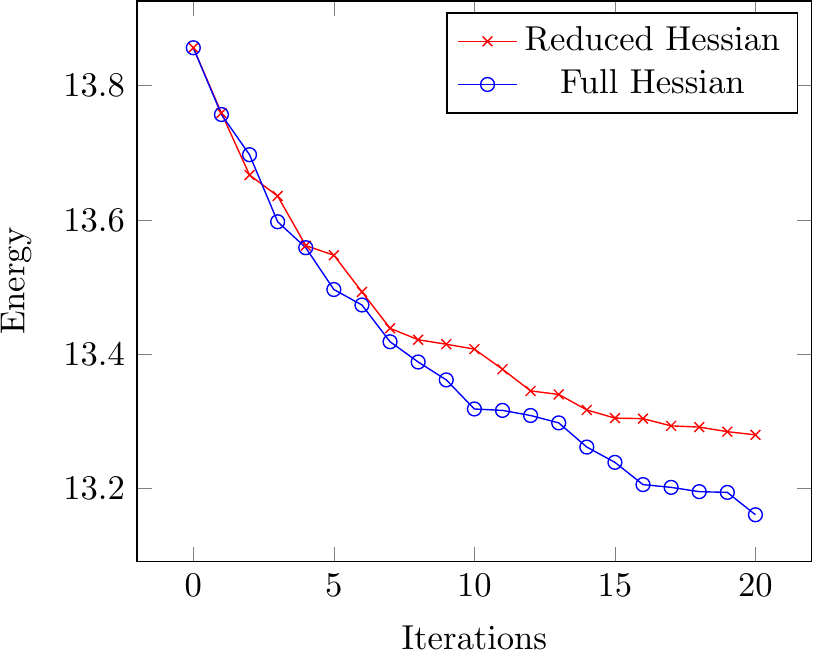}
\caption{Energy decay in Newton's method using the full Hessian and the approximation $\partial^2 E^d_n \approx H^\loc + \partial^2 
R_n$.}\label{Fig:NewtonAndReducedNewton}\end{subfigure} &
\begin{subfigure}[t]{0.46\textwidth}
\begin{tikzpicture}[nodes = {font=\sffamily}, scale=0.75, transform shape]
  \foreach \percent/\name in {
      75.6/Solver,
      13.6/Computation of the Hessian object,
      6.4/Armijo line search,
      4.1/Computation of the gradient object,
      0.3/Other
    } {
      \ifx\percent\empty\else               
        \global\advance\cyclecount by 1     
        \global\advance\ind by 1            
        \ifnum3<\cyclecount                 
          \global\cyclecount=0              
          \global\ind=0                     
        \fi
        \pgfmathparse{\cyclelist[\the\ind]} 
        \edef\color{\pgfmathresult}         
        \draw[fill={\color!50},draw={\color}] (0,0) -- (\angle:\radius)
          arc (\angle:\angle+\percent*3.6:\radius) -- cycle;
        \node at (\angle+0.5*\percent*3.6:0.7*\radius) {\percent\,\%};
        \node[pin=\angle+0.5*\percent*3.6:\name]
          at (\angle+0.5*\percent*3.6:\radius) {};
        \pgfmathparse{\angle+\percent*3.6}  
        \xdef\angle{\pgfmathresult}         
      \fi
    };
\end{tikzpicture}
\caption{Computational costs of each step of Newton's method.}\label{fig:CompComplexity}
\end{subfigure}
\end{tabular}
\caption{ Performance of Newton's method using a fully and partially assembled Hessian operator (left) and a breakdown of the computational costs of a generic step of Newton's method (right). }
\label{Fig:NewtonAndVariantAndComplexity}
\end{figure}
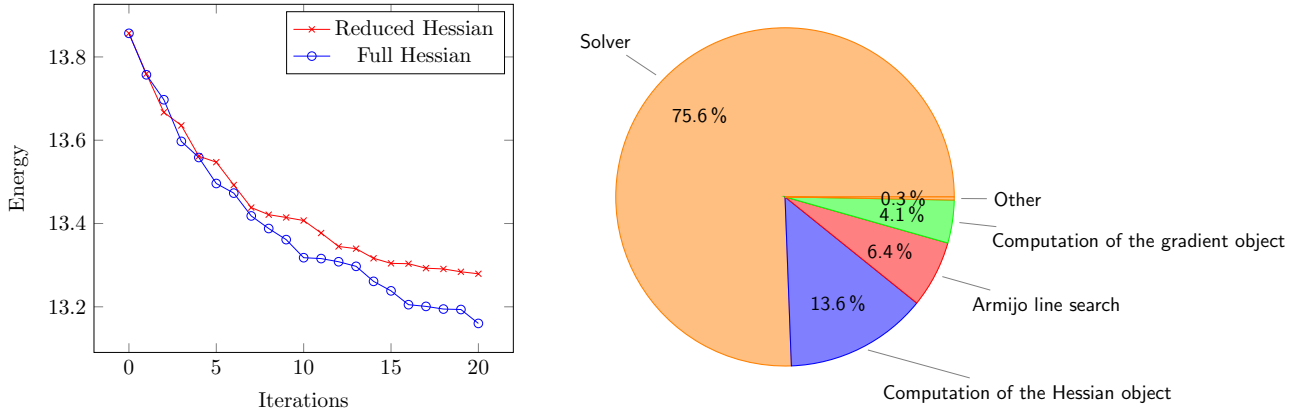

Further modifications of the Newton iteration allow a further reduction of the computational complexity of each Newton step. Since $\partial_1\d(\F(\imgVol\circ y), 
\imgPln)$ is zero for perfectly aligned images, the contribution of $H^\loc$ for closely aligned images is negligible and $\partial^2 J^d_n$ can be 
approximated by $H^\nonloc + \partial^2R_n$.
We refer to \cite{ModersitzkiRuthottoGreif} for a detailed discussion of this strategy in the context of hyperelastic 3D-3D image registration.
\Cref{Fig:NewtonAndReducedNewton} compares the energy decrease of Newton's method and this 
modification when applied to the datasets shown in \Cref{Fig:SyntheticVesselStructure}.

\Cref{fig:CompComplexity} shows that the major cost of each Newton iteration lies in the numerical solution of the linear system.
To improve convergence of the BiCGStab solver, we tested several preconditioning methods. Jacobi, geometric scaling, and incomplete LU preconditioning -- applied to the part
$H^\loc[y_k^n] + \partial^2 R_n[y_k^n]$ of the Hessian matrix that can be assembled -- turned out to be inferior to a multigrid preconditioner 
applied to the entire Hessian operator $\partial^2 E^d_n$. 
Note that this operation does not require the assembly of $\partial^2 E^d_n$, since iterative solvers like BiCGStab or GMRES -- this time without preconditioning -- 
can be employed for the pre- and postsmoothing steps.

A line search Newton method is prone to getting stuck or at least slowing down at saddle points
(even despite the compensation for indefiniteness).
This can be avoided by using a trust region Newton method,
which can also minimize indefinite quadratic functions within its trust region.
To this end we solved the Newton system via preconditioned truncated Lanczos iterations \cite{gould1999solving, zhang2017generalized}
(since this simultaneously allows to use the above-mentioned technique for compensating indefiniteness),
where we applied the same preconditioners as in the line search approach.

Overall, as already suggested by \cref{fig:CompComplexity},
the solution of the linear system in Newton's method (line search or trust region) turns out to consume so much time
that a mere quasi-Newton method with BFGS updates (see e.\,g.~\cite[Ch.\,6.1]{nocedal2006numerical}) is more efficient.
In fact, a plain conjugate gradient descent with Polak--Ribi\'ere updates (see e.\,g.~\cite[Sct.\,8.5]{ciarlet1989introduction}) performed best in our experiments.

\paragraph{Preprocessing.}
Before starting the minimization of $E^d_n$ we rigidly align $\imgVol$ to $\imgPln$ by minimizing $J^d_n[y^n]$ among all rigid deformations (for which actually $E^d_n=J^d_n$).
In fact, to accommodate a potential slight mismatch in magnification between the 2D intravital and the 3D confocal microscopy as well as different resolutions in $x_1$-, $x_2$-, and $x_3$-direction (or to make up for a bad choice of the blurring kernel $\chi$),
we additionally allow a rescaling along the coordinate directions.
Thus, we minimize $J^d_n$ among all deformations
\begin{align*}
 x & \mapsto R_3(\gamma)\,R_2(\beta)\,R_1(\alpha)\,\mathrm{diag}(s_1, s_2, s_3)\, x + t\, ,
\end{align*}
parameterized by a translation vector $t\in\R^3$ as well as scalings $s_1,s_2,s_3$ along and rotation angles $\alpha,\beta,\gamma\in [0, 2\pi[$ about the three coordinate directions
($R_i(\delta)\in SO(3)$ denotes the rotation about the $i$\textsuperscript{th} axis by angle $\delta$).
We now use the grid hierarchy $\TVol^n$, $n=1,\ldots,N$, to iteratively find the optimal parameters for each grid level $n$
via a quasi-Newton method initialized with the optimal parameters from level $n-1$.
After the optimal deformation on level $N$ is found, we replace $\imgVol$ by its composition with that deformation
so that the new $\imgVol$ now has the correct length scales and already is optimally aligned.

\paragraph{Multilevel optimization problems.}
We have already detailed how by replacing $\J^d$, $\Reg$, and $\E^d$ with discretized analogues $J^{d}_n$, $E^d_n$, and $R_n$ we arrive at a set of optimization problems
\begin{align*}
  \min_{y^n\in (X_{\threeD}^n)^3} E^{d}_n[y^n]\,, \quad n=1, \ldots, N\, ,
\end{align*}
that are numerically solved using the nonlinear conjugate gradient method.
As so often in image registration methods, the use of a multilevel approach is essential for the quality of the results as well as for computational efficiency.
Owing to its nonconvexity, the functional $\E^d$ can be expected to have a large number of local minima, which is in general linked to
the resolution of the given image data and the number of image features that promote regional alignment.
Downsampling of the image data reduces the number of image features (and thus of local minima) in the input datasets as well as
the complexity of the optimization procedure (see the experimental illustration in \cref{fig:multiscale}). The results of the less costly optimization on coarser grids can then be used as good initial 
values for the higher level optimization problems.
We make use of this strategy by first minimizing $E^d_n$ on a low grid level $n$ and then successively solving the optimization problem on higher grid levels.

In addition to the use of hierarchical grids we will use a smoothing-based multiscale strategy.
A drawback of local, pixel-based distance measures is their inability to align corresponding but non-overlapping image features.
As illustrated in \cref{fig:multiscale}, blurring of the datasets
can compensate the lack of overlap at the expense of a diminished data fidelity.
We therefore solve the registration problem for blurred versions of $\imgPln$ and $\imgVol$ with successively decreasing blur radius.
In our implementation, we convolved $\imgPln$ and the 
slices of $\imgVol$ in $x_1$-$x_2$-direction with a Gaussian kernel $K_s(x) = \frac{1}{s\sqrt{2\pi}} \exp(-\frac{x^2}{2s^2})$ via the fast Fourier transform.
Experimentally, a good sequence of decreasing kernel radii turned out to be $s = 8h, 4h, 2h, 0$, where $h$ denotes the grid width of the 
volumetric input image.

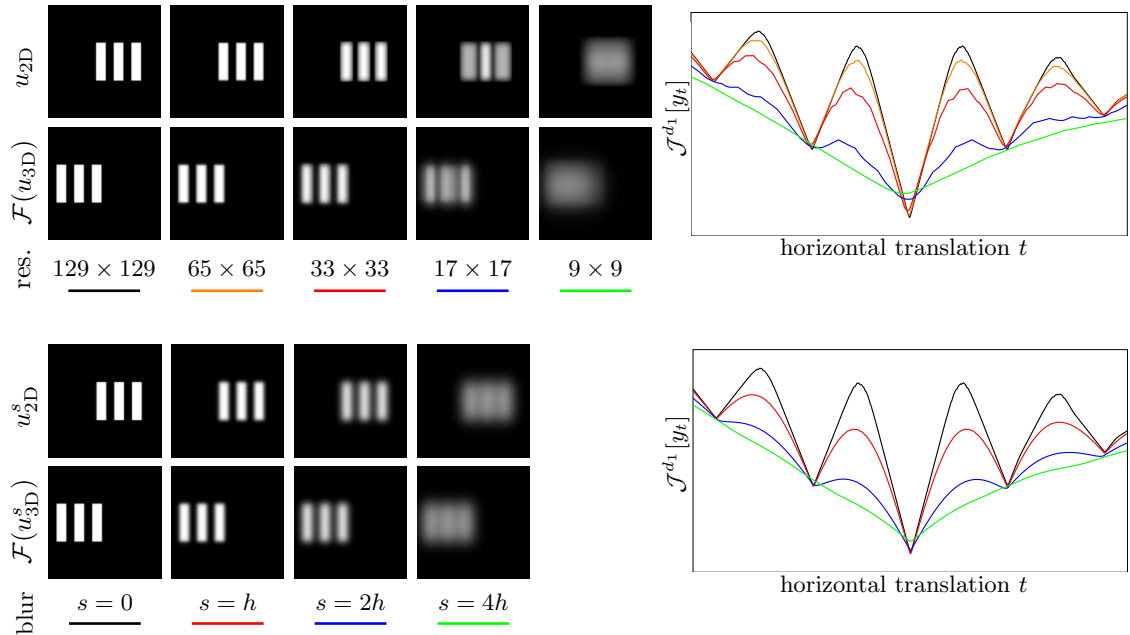
\begin{figure}
\setlength\unitlength{.1\linewidth}
\setlength\tabcolsep{2pt}
\begin{tabular}{cccccc}
\rotatebox{90}{$\;\;\;\;\imgPln$}&
\includegraphics[width=\unitlength]{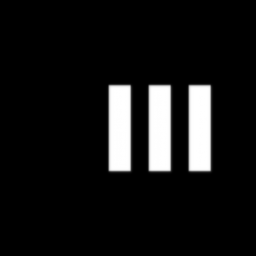} &
\includegraphics[width=\unitlength]{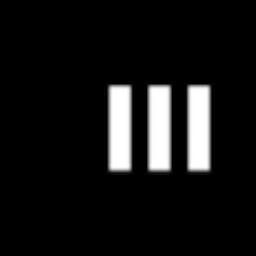} &
\includegraphics[width=\unitlength]{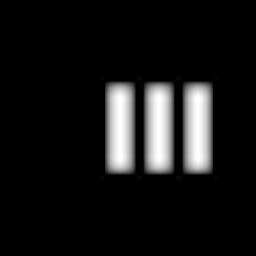} &
\includegraphics[width=\unitlength]{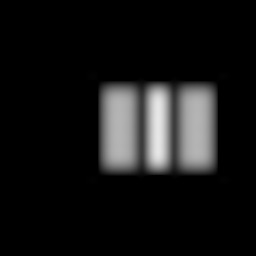} &
\includegraphics[width=\unitlength]{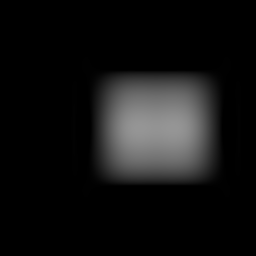} \\
\rotatebox{90}{$\;\;\F(\imgVol)$}&
\includegraphics[width=\unitlength]{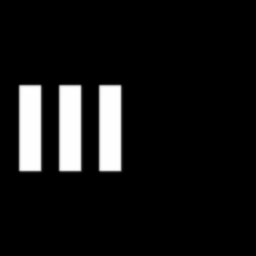} &
\includegraphics[width=\unitlength]{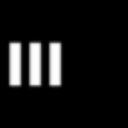} &
\includegraphics[width=\unitlength]{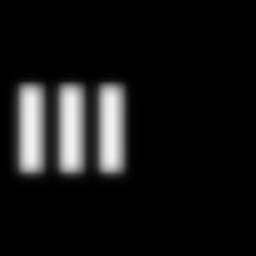} &
\includegraphics[width=\unitlength]{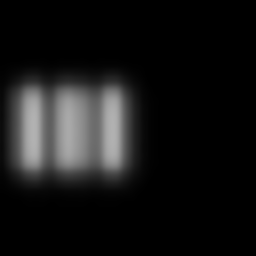} &
\includegraphics[width=\unitlength]{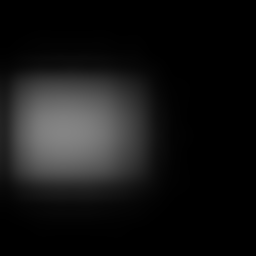} \\
\rotatebox{90}{\!\!\!res.}&
\small$\underset{\rule{7ex}{1pt}}{129\times129}$&
\small$\underset{\color{orange}\rule{7ex}{1pt}}{65\times65}$&
\small$\underset{\color{red}\rule{7ex}{1pt}}{33\times33}$&
\small$\underset{\color{blue}\rule{7ex}{1pt}}{17\times17}$&
\small$\underset{\color{green}\rule{7ex}{1pt}}{9\times9}$
\end{tabular}
\hfill
\begin{picture}(4,0)(0,.8)
\put(0,0){\includegraphics[width=4\unitlength]{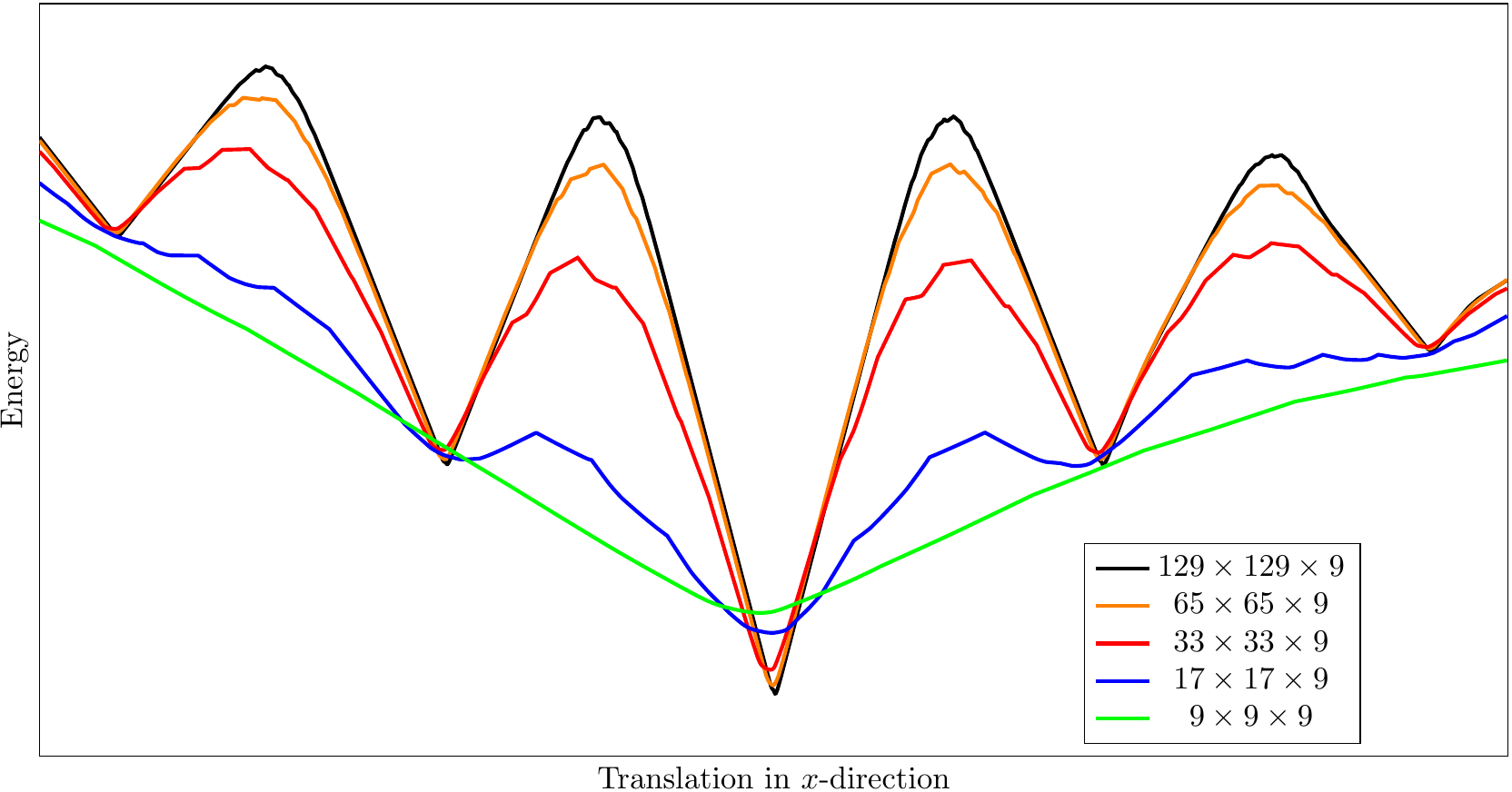}}
\put(0,0){\color{white}\rule{4\unitlength}{.9ex}}
\put(0,0){\color{white}\rule{1ex}{2\unitlength}}
\put(2.8,.11){\color{white}\rule{1\unitlength}{.6\unitlength}}
\put(0,-.08){\makebox[4\unitlength][c]{\small horizontal translation $t$}}
\put(-.15,.8){\rotatebox{90}{\small$\J^{d_1}[y_t]$}}
\end{picture}
\\[1.5\baselineskip]
\begin{tabular}{cccccc}
\rotatebox{90}{$\;\;\;\;\imgPln^s$}&
\includegraphics[width=\unitlength]{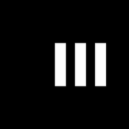} &
\includegraphics[width=\unitlength]{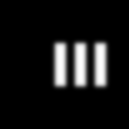} &
\includegraphics[width=\unitlength]{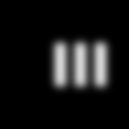} &
\includegraphics[width=\unitlength]{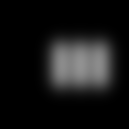} \\
\rotatebox{90}{$\;\;\F(\imgVol^s)$}&
\includegraphics[width=\unitlength]{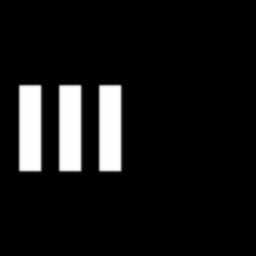} &
\includegraphics[width=\unitlength]{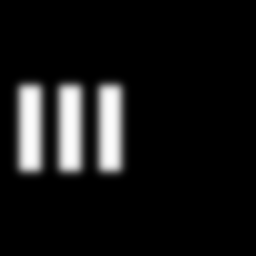} &
\includegraphics[width=\unitlength]{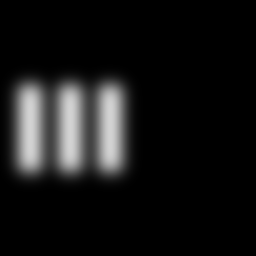} &
\includegraphics[width=\unitlength]{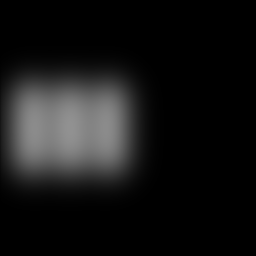} \\
\rotatebox{90}{\small\!\!\!\!\!\!blur}&
\small$\underset{\rule{7ex}{1pt}}{s=0}$&
\small$\underset{\color{red}\rule{7ex}{1pt}}{s=h}$&
\small$\underset{\color{blue}\rule{7ex}{1pt}}{s=2h}$&
\small$\underset{\color{green}\rule{7ex}{1pt}}{s=4h}$
\end{tabular}
\hfill
\begin{picture}(4,0)(0,.8)
\put(0,0){\includegraphics[width=4\unitlength]{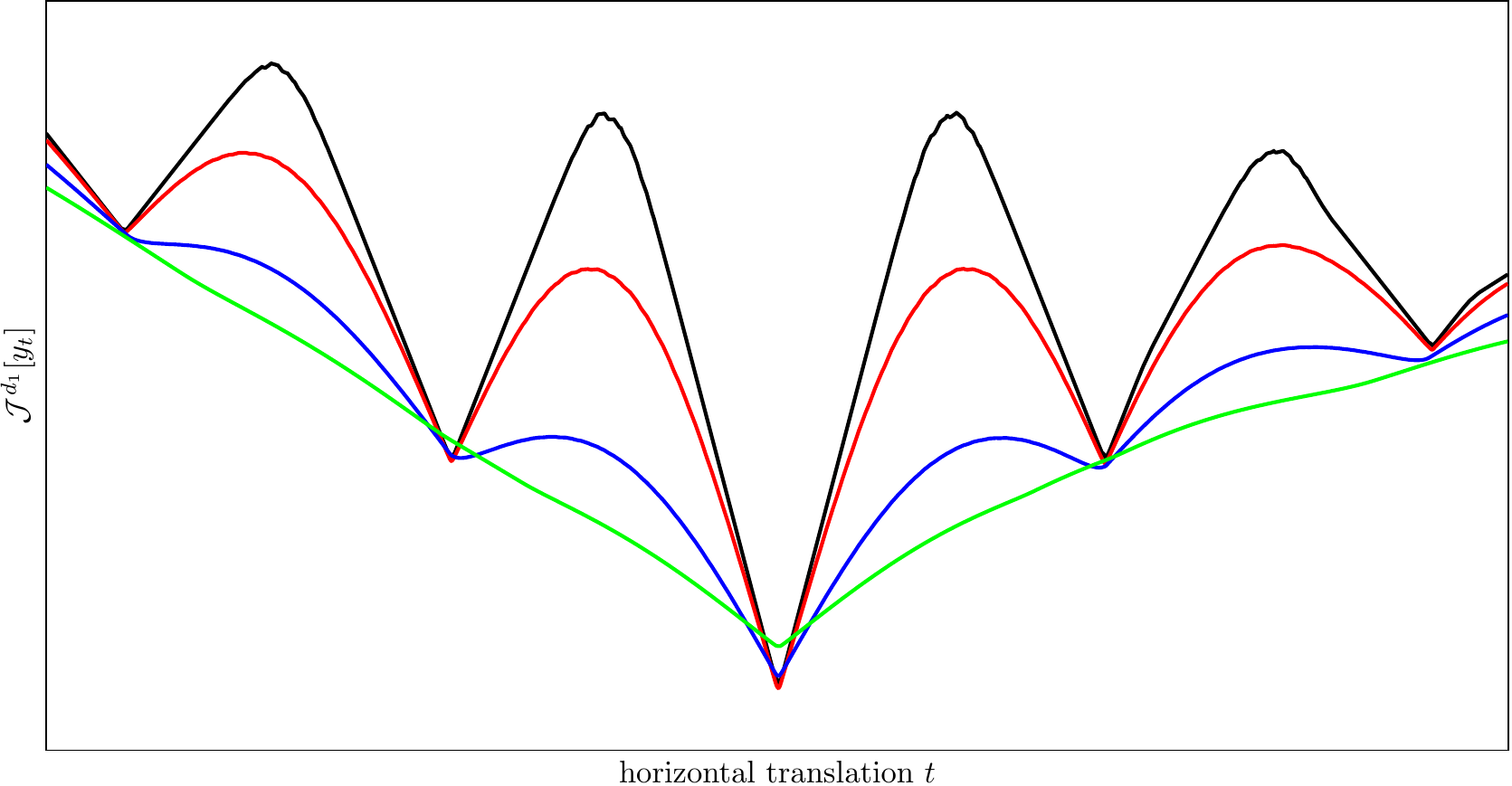}}
\put(0,0){\color{white}\rule{4\unitlength}{.9ex}}
\put(0,0){\color{white}\rule{1ex}{2\unitlength}}
\put(2.8,.11){\color{white}\rule{1\unitlength}{.6\unitlength}}
\put(0,-.08){\makebox[4\unitlength][c]{\small horizontal translation $t$}}
\put(-.15,.8){\rotatebox{90}{\small$\J^{d_1}[y_t]$}}
\end{picture}
\caption{Multigrid and multiscale strategies simplify the energy landscape and thereby facilitate registration:
Both top and bottom experiment use a three-dimensional image $\imgVol$ consisting of three cuboids at a resolution of $256\times256\times9$ pixels
and take the two-dimensional image $\imgPln$ as the projection of a simple translation, $\imgPln=\F(\imgVol\circ y_{0.35})$ for $y_t(x)=(x_1,x_2,x_3+t)$
(note that the image width of $256$ pixels corresponds to length $1$).
The right graphs show the registration energy as a function of the translation $t$ in the range $[0,0.7]$,
where the different colours correspond to the different resolutions or blurs shown on the left.
In the top experiment, the $x_1,x_2$-resolution of both $\imgVol$ and $\imgPln$ was decreased repeatedly using the restriction operators $\mathfrak R_n$,
while in the bottom experiment the images are obtained by blurring in $x_1,x_2$-direction with the Gaussian kernel $K_s$ of scale $s$, $\imgPln^s=K_s*\imgPln$, $\imgVol^s=K_s*\twoD\imgVol$.
Clearly, spurious local minima are alleviated or even eliminated at coarser resolution or stronger blur.
}
\label{fig:multiscale}
\end{figure}

\begin{figure}
\centering
\setlength\unitlength{.1\linewidth}
\begin{picture}(5,1)
\put(0,0){\includegraphics[width=2\unitlength]{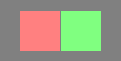}}
\put(3,0){\includegraphics[width=2\unitlength]{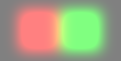}}
\multiput(4,.3)(0,.1){5}{\vector(1,0){.2}}
\end{picture}
\caption{Left: A reference image $u_r$ (red) and a template image $u_t$ (green) show the same structure at different locations.
Due to the lack of overlap, the gradient of the registration functional $y\mapsto\int_{\R^2}d(u_t\circ y,u_r)\,\d x$ at $y=\id$ is zero.
Right: After blurring both images the structures overlap, resulting in a nonzero gradient.
The negative gradient, which points into the direction of an improved registration, can be interpreted as a perturbation of $y=\id$
which produces a slight rightward deformation in the overlapping region, as indicated by the arrows.}
\label{fig:multiscale}
\end{figure}

\paragraph{Full algorithm.}
The full algorithmic workflow is depicted in \cref{Fig:Workflow}.
Note that the iteration over the different grid levels always starts with that level $n$ as the coarsest one on which the grid size corresponds to the current blurring kernel radius.
The actual implementation was based on the \emph{QuocMesh Library}, a C\texttt{++} Finite Element library which supports quadratic, cuboid and simplicial elements.

\begin{figure}
\definecolor{ColBlock}{HTML}{8BA0AB}    
\definecolor{ColDecision}{HTML}{294B5E}    
\definecolor{ColCloud}{HTML}{16394C}
\centering

\tikzstyle{decision} = [diamond, draw, fill={ColDecision}, text width=6.5em, text badly centered, node distance=3cm, inner sep=0pt, text=white]
\tikzstyle{line} = [draw, -latex']
\tikzstyle{cloud} = [draw, ellipse, fill={ColCloud}, node distance=6cm, minimum height=3em, text=white]

\tikzstyle{block} = [rectangle, draw, fill={ColBlock}, text width=12em, text centered, rounded corners, minimum height=4em, text=white]

\scalebox{0.825}{
\begin{tikzpicture}[node distance = 4cm, auto]
    \node [block] (init) {Create hierarchical image data representation};
    \node [block, below of=init, node distance=3.5cm] (preprocessing) {Preprocessing (rigid registration + scaling)};
    \node [cloud, below of=preprocessing, node distance=2.5cm, text width = 8.5em, align=center] (rigid_out) {Rigid intermediate result};
    \node [block, right of=preprocessing, text width=7.5em,  node distance=5cm] (start_elastic) {Elastic registration};
    \node [decision, right of=start_elastic, node distance=4.0cm] (decide) {Artificial blurring involved in previous step?};
    \node [block, above of=decide, node distance=3.5cm,  text width=6.5em] (decrease_radius) {Decrease artificial blur};
    \node [block, right of=decide, node distance=3.5cm, text width=4.5em] (stop) {Stop};
    \node [cloud, below of=stop, node distance=2.5cm, align=center, text width=5.5em] (final_out) {Registration result};
    \path [line] (init) -- (preprocessing);
    \path [line] (preprocessing) -- node[below=0.75cm] {apply computed deformation} (start_elastic);
    \path [line] (start_elastic) -- (decide);
    \path [line] (decide) -- node {no}(stop);
    \path [line, thick, dashed] (preprocessing) -- (rigid_out);
    \path [line, thick, dashed] (stop) -- (final_out);
    
    \Loop[dist=1.5cm,dir=NO, label=\mbox{iterate over grid levels},labelstyle=above left](preprocessing);
    \Loop[dist=1.5cm,dir=NO, label=\mbox{iterate over grid levels},labelstyle=above left](start_elastic);
    \path [line] (decide) --  node {yes} (decrease_radius);
    \path [line] (decrease_radius) -| node[above] {update initial deformation} (start_elastic);
\end{tikzpicture}
}
\caption{Schematic overview of the overall registration procedure.}
\label{Fig:Workflow}
\end{figure}
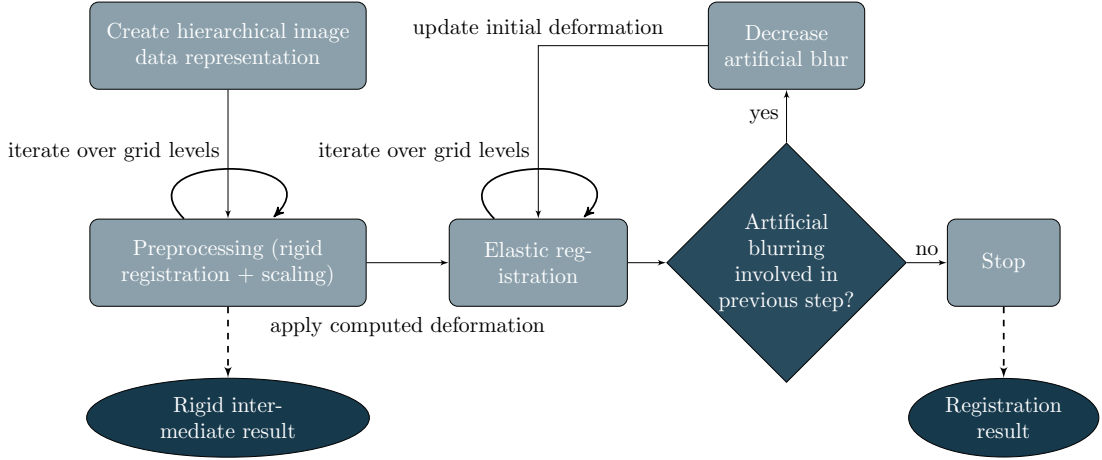

\section{Experimental results}\label{sec:results}

In all following examples we took $\domainVol= [0,1]^2 \times [0,Z]$ for some $Z \in ]0, 1]$ and used the blurring kernel
\begin{align*}
  \chi(x) = \begin{cases}
                (\pi\, [x_3-\frac{Z}{2}]^2)^{-1} & \sqrt{x_1^2+x_2^2} \leq \left|x_3 - \frac{Z}{2}\right|, \\
                0 & \text{ else,}
               \end{cases}
\end{align*}which corresponds geometrically to a double cone.

\paragraph{Synthetic data.} To test the performance of the elastic regularizer, we applied our technique to synthetic datasets (\crefrange{Fig:ToyExampleInplane}{Fig:SyntheticVesselStructure}), representing cuboids and a deformed vessel structure made of some elastic material.

In the first two test cases (\crefrange{Fig:ToyExampleInplane}{Fig:ToyExampleZDislocation}), the simplicity of the shapes made it possible to generate 
the three-dimensional deformed and undeformed scenes $\imgVol$ and $\imgVol\circ y$ by setting the pixel intensities manually. The two-dimensional reference images $\imgPln$ were then obtained by applying the forward 
operator to the undeformed scenes. \Cref{Fig:ToyExampleInplane} depicts the results of a test assessing the algorithm's ability to compute 
non-rigid lateral deformations. The resolutions of $\imgVol$ and $\imgPln$ in this example are $129\times 129 \times 17$ and $129 \times 129$. \Cref{Fig:ToyExampleInplane} shows, that the overall structural 
alignment works flawlessly, but also that small-scale spurious deformations can be introduced locally which have negligible influence on the data fidelity term.

\begin{figure}
\center
\begin{tabular}{p{.23\textwidth}p{.23\textwidth}p{.23\textwidth}p{.23\textwidth}}
\includegraphics[width=.23\textwidth]{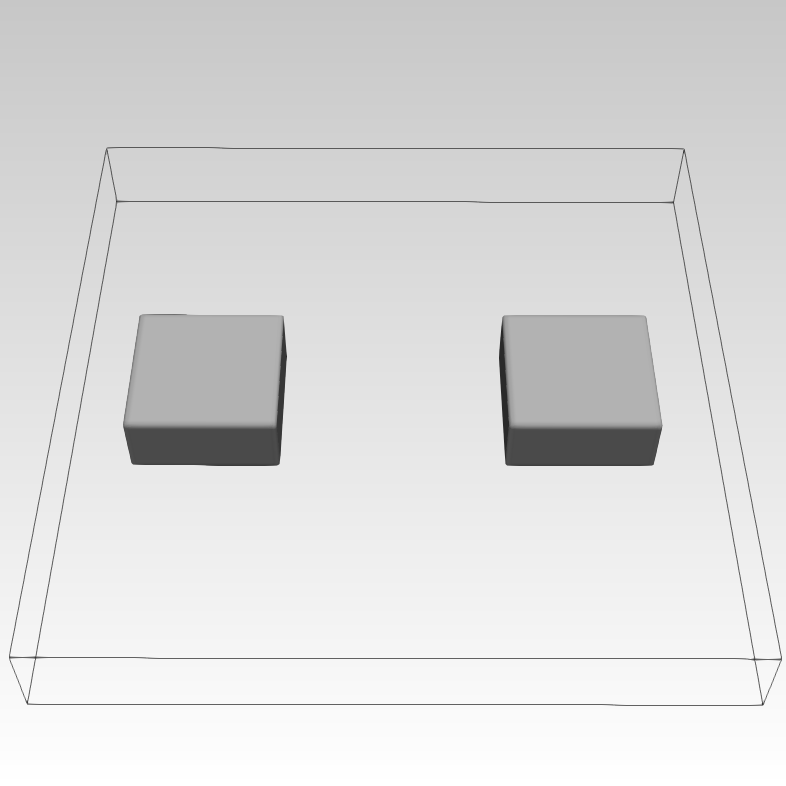}
\mbox{Initial\,configuration\,$\imgVol$} (perspective view) &
\includegraphics[width=.23\textwidth]{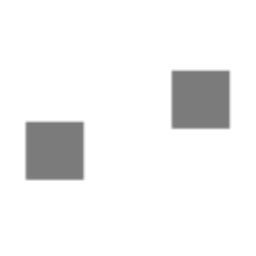}
Two-dimensional reference image $\imgPln$ &
\includegraphics[width=.23\textwidth]{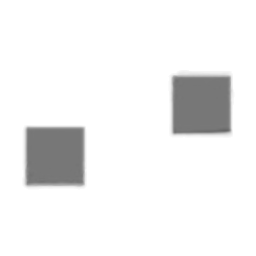}
Projected registration result $\F(\imgVol\circ y)$ &
\includegraphics[width=.23\textwidth]{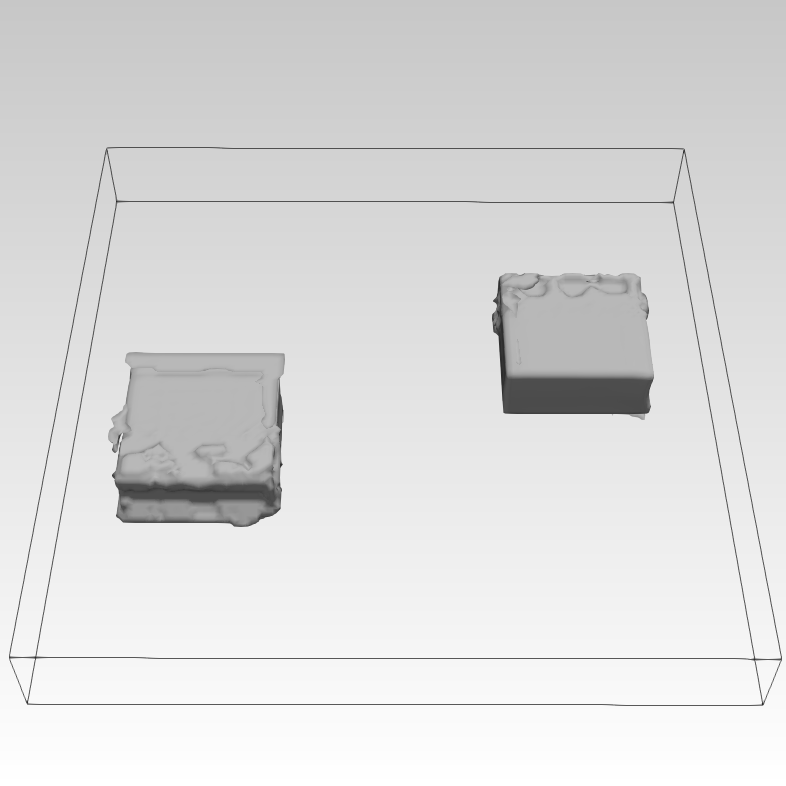}
\mbox{Registration\,result\,$\imgVol\!\circ\!y$} (perspective view)
\end{tabular}
\caption{Inplane deformation of a pair of elastic cuboids. In this experiment and the one shown in \cref{Fig:ToyExampleZDislocation}, the parameters of the stored energy function $W$ were chosen to make the two cuboids
behave like they were embedded in some soft, gel-like substance.}
\label{Fig:ToyExampleInplane}
\end{figure}

The second test, shown in \cref{Fig:ToyExampleZDislocation}, is intended to evaluate how well the algorithm infers the displacement in viewing direction from the blurriness of the two-dimensional image.
To this end we use a configuration of two cubes, where the deformed scene differs from the 
undeformed configuration only by a vertical displacement of one cube.
The underlying resolutions are $129\times 129 \times 129$ and $129 \times 129$, respectively. As \cref{Fig:ToyExampleZDislocation} shows, the left cube is correctly displaced in viewing direction, but its initial shape is 
not entirely preserved. The lack of smoothness of the actually applied transformation
\begin{align*}
 x &\mapsto \begin{cases}
                   \left(x_1,x_2,x_3+\frac{1}{3}\right), &\text{if}\  0 \leq x_1 \leq \frac{1}{2}\\
                   x,& \mathrm{else,}
                  \end{cases}
\end{align*}
which is incompatible with the regularizer $\Reg$,
explains this phenomenon. As the hyperelastic regularizer favours more regular deformations, its contribution to the overall energy will outweigh those of the data term, if, as in this case, a transformation introduces too much shear.

\begin{figure}
\center
\begin{tabular}{p{.23\textwidth}p{.23\textwidth}p{.23\textwidth}p{.23\textwidth}}
\includegraphics[width=.23\textwidth]{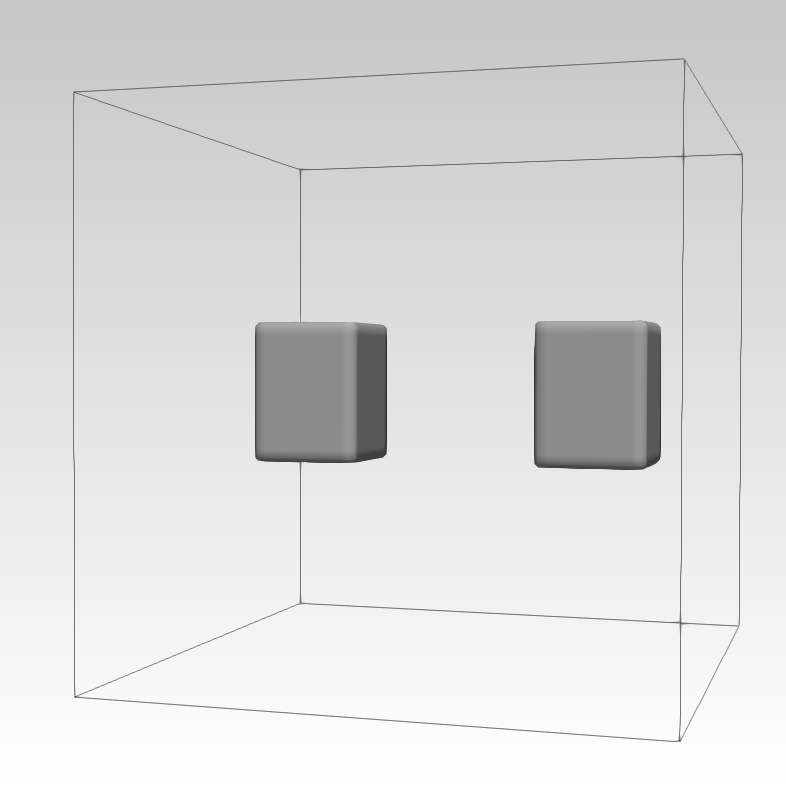}
\mbox{Initial\,configuration\,$\imgVol$} (side view) &
\includegraphics[width=.23\textwidth]{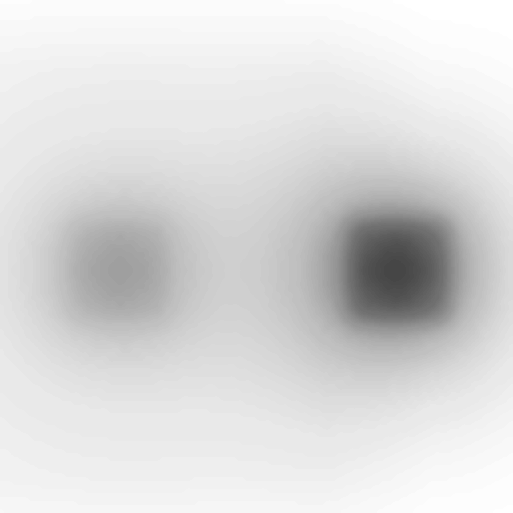}
Two-dimensional reference image $\imgPln$ &
\includegraphics[width=.23\textwidth]{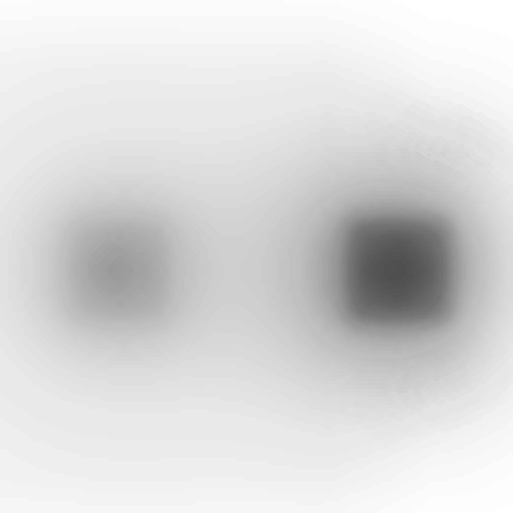}
Projected registration result $\F(\imgVol\circ y)$ &
\includegraphics[width=.23\textwidth]{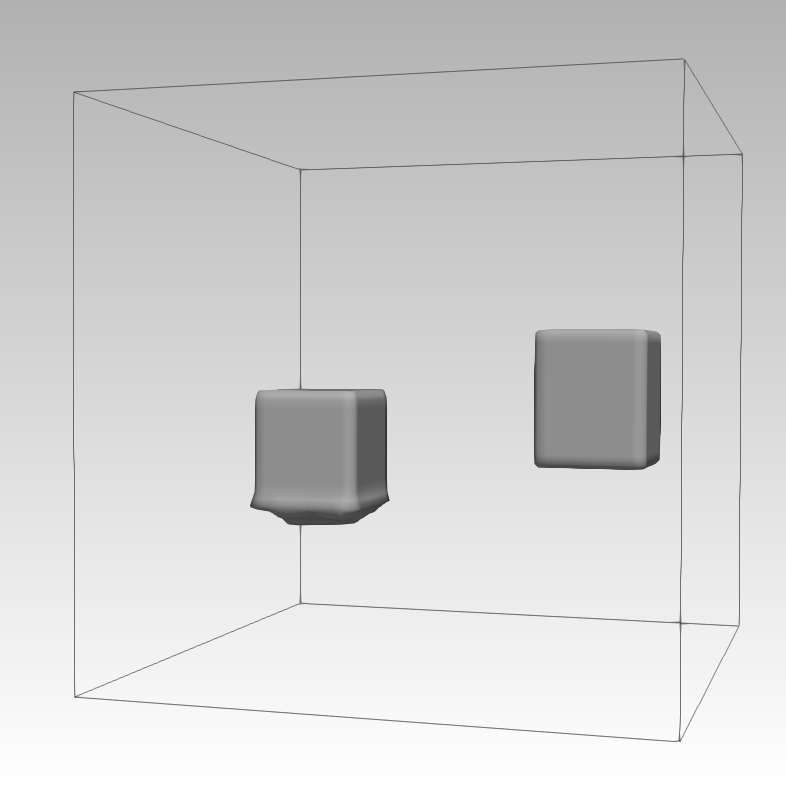}
\mbox{Registration\,result\,$\imgVol\!\circ\!y$} (side view)
\end{tabular}
\caption{Displacement of one of a pair of elastic cubes, surrounded by a soft material.
}
\label{Fig:ToyExampleZDislocation}
\end{figure}

Since vessel structures constitute the predominant image features in our microscope images, we apply the algorithm to another, more realistic test case to see 
whether branched tube-like structures are registered equally well as in the previous cases. The synthetic dataset shown in 
\cref{Fig:SyntheticVesselStructure} was generated with the help of \texttt{VascuSynth} \cite{VascuSynth2010}, a software package 
capable of generating realistically looking synthetic vessel structures. To obtain a volumetric template dataset $\imgVol$, the generated vessel structure was deformed 
using a CGAL \cite{Ref:CGAL} implementation of the algorithm described in \cite{Sorkine2007}, which is capable of generating triangular 
mesh deformations in real-time under the constraint that the resulting deformation acts as rigidly as possible on each triangle.
The deformed and undeformed triangular meshes were then turned into 
grayscale image stacks. As in the previous tests, the application of the forward operator to the undeformed volume image stack generated 
the two-dimensional reference image $\imgPln$. The dimensions of the input datasets were the same as in our first test case. It turned out that in this example, 
a preprocessing step as described in \cref{sec:numerics}, preceding the elastic registration procedure, was necessary to obtain a satisfactory data alignment. 
The comparison of the reference image $\imgPln$ and the projected registration result $\F(\imgVol\circ y)$ in \cref{Fig:SyntheticVesselStructure} indicates a faultless overall alignment, but 
spurious small-scale deformations similar to those encountered in the first test case.

\begin{figure}
\center
\begin{tabular}{p{.23\textwidth}p{.23\textwidth}p{.23\textwidth}p{.23\textwidth}}
\includegraphics[width=.23\textwidth]{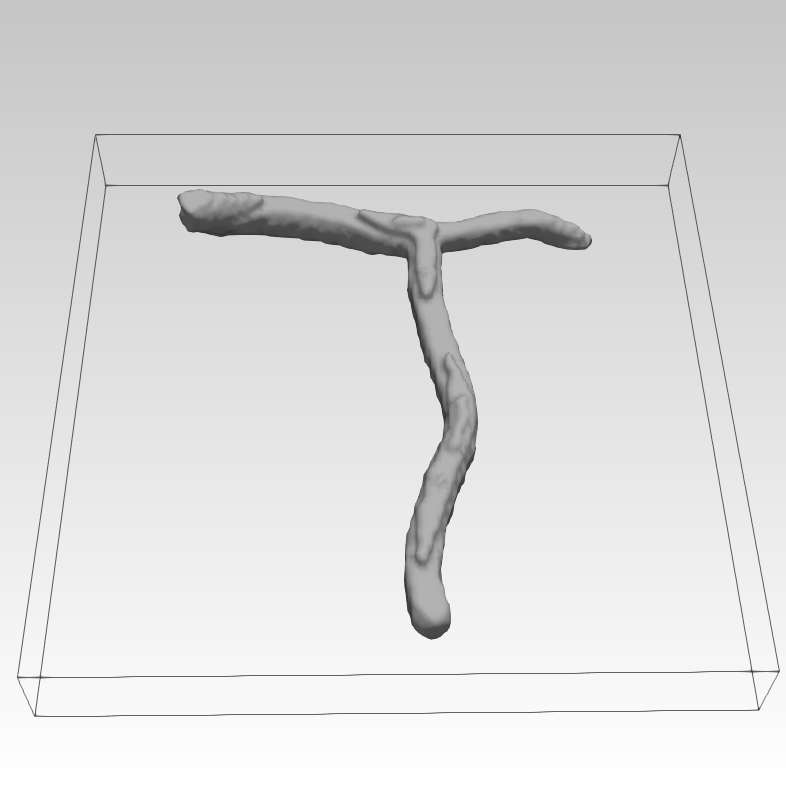}
\mbox{Initial\,configuration\,$\imgVol$} (perspective view) &
\includegraphics[width=.23\textwidth]{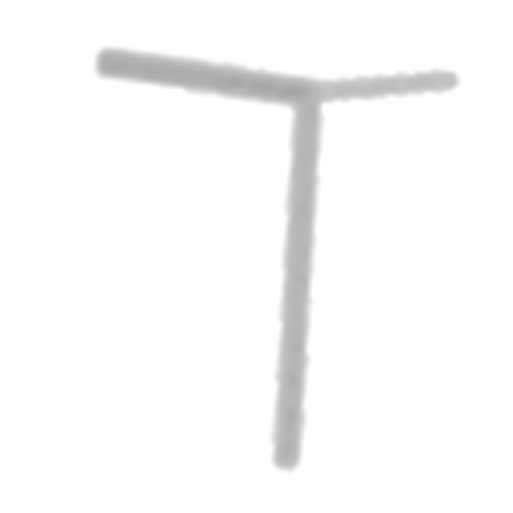}
Two-dimensional reference image $\imgPln$ &
\includegraphics[width=.23\textwidth]{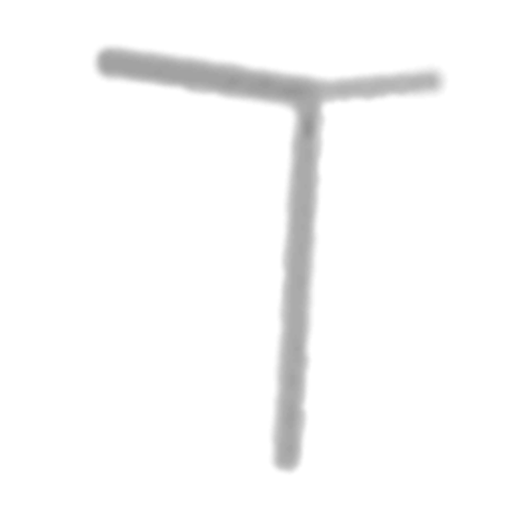}
Projected registration result $\F(\imgVol\circ y)$ &
\includegraphics[width=.23\textwidth]{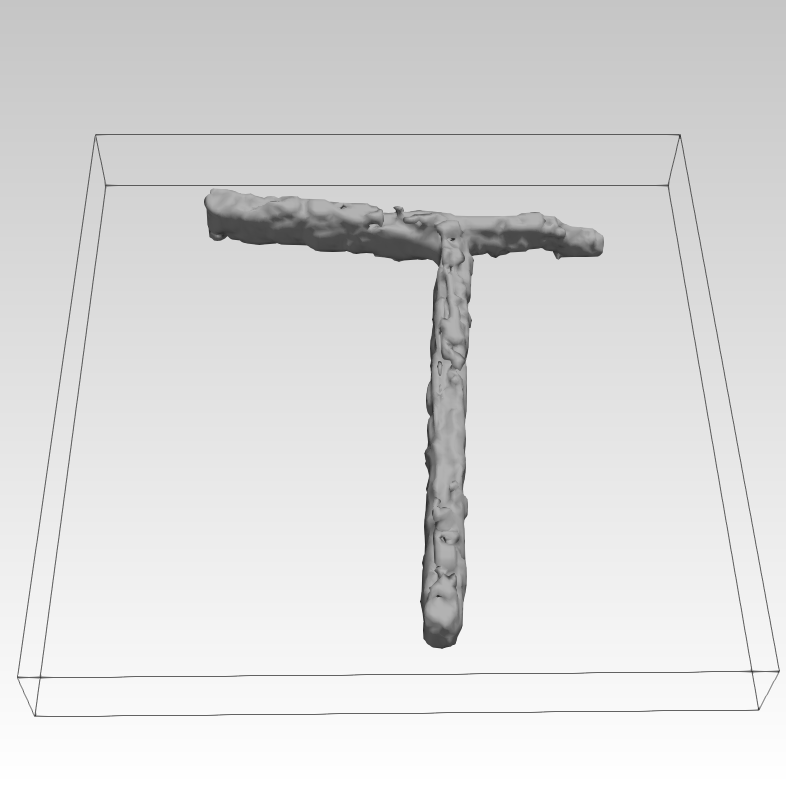}
\mbox{Registration\,result\,$\imgVol\!\circ\!y$} (perspective view)
\end{tabular}
\caption{Idealized vessel structure used as a more realistic test case.}
\label{Fig:SyntheticVesselStructure}
\end{figure}

\paragraph{Microscopy data.}
In \cref{Fig:MicroscopyImages} the technique was finally applied to a microscopy dataset as described in the introduction.

\begin{figure}
\center
\begin{tabular}{p{.23\textwidth}p{.23\textwidth}p{.23\textwidth}p{.23\textwidth}}
\includegraphics[width=.23\textwidth]{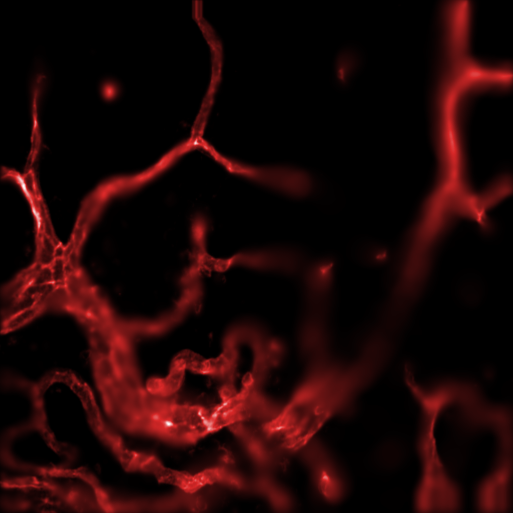}
Projected template image $\imgVol$ &
\includegraphics[width=.23\textwidth]{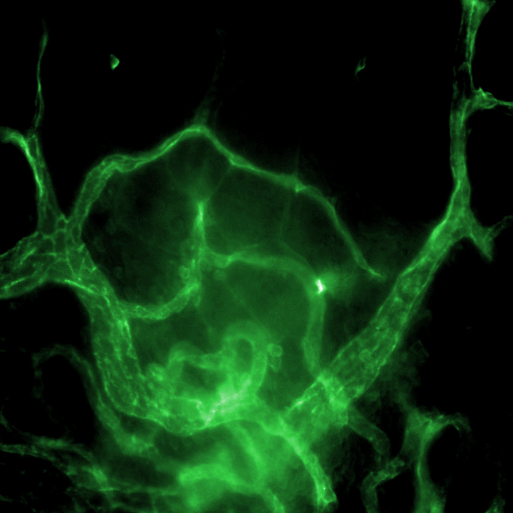}
Two-dimensional reference image $\imgPln$ &
\includegraphics[width=.23\textwidth]{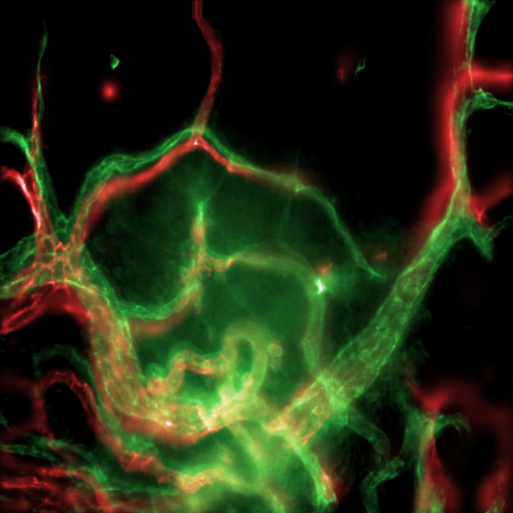}
Overlay of $\imgVol$ and $\imgPln$ &
\includegraphics[width=.23\textwidth]{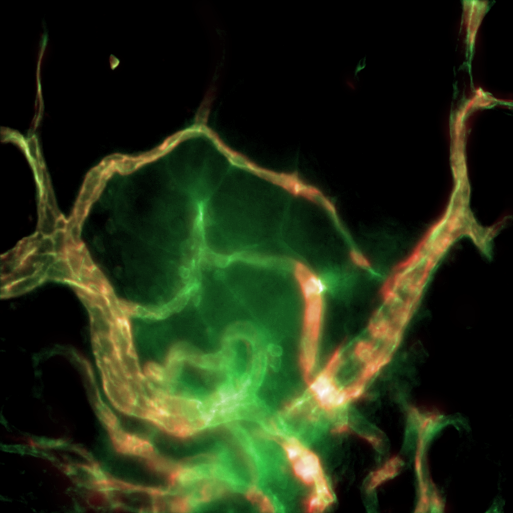}
Overall registration result
\end{tabular}
\caption{Microscopy images acquired with 3D confocal and 2D IVM microscopy and their elastic registration. Data courtesy of Lydia Sorokin, Konrad Buscher, Jian Song (Institute of Physiological Chemistry and Pathobiochemistry, M\"unster, Germany).}
\label{Fig:MicroscopyImages}
\end{figure}

\begin{figure}
\center
\includegraphics[width=.23\textwidth]{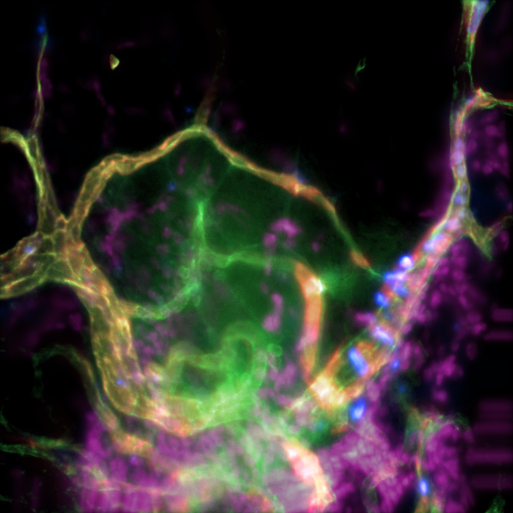}
\caption{Registration result from \cref{Fig:MicroscopyImages} showing additional channels
combining temporal information from two-dimensional intravital microscopy
with well-resolved spatial information from three-dimensional confocal microscopy images acquired after tissue excision.
Magenta shows fat cells (from confocal microscopy), blue shows leukocytes (from intravital microscopy).}
\label{fig:finalGoal}
\end{figure}

The centre region of the reference image was obscured by diffused fluorescence dye, interfering with the registration process. As a consequence, the data term
had to be augmented by a mask $m:\domainPln \to (0,1]$ taking small values in the degraded image region,
\begin{align*}
 \J^d[y] =\int_{\domainPln}m(x)\, d(\F( \imgVol \circ y)(x), \imgPln(x) ) \ \d x\,.
\end{align*}
As can be seen from overlaying $\imgPln$ and $\F(\imgVol\circ y)$ in \cref{Fig:MicroscopyImages}, right, the \notinclude{in-plane }alignment of the blood vessel structures is satisfactory.
Note that projecting the original, undeformed three-dimensional configuration, as shown in \cref{Fig:MicroscopyImages} left,
one obtains dark regions on the right middle part of the image, where vessels lie outside the focus plane.
This is corrected by the registration so that the overall registration result on the right of \cref{Fig:MicroscopyImages} shows both a strong red and green signal in that region.
Similarly, the in-plane distortion on the left side of the image is corrected by the registration.
The alignment of the three-dimensional with the two-dimensional images
allows information of other colour channels to be integrated into the single dataset as shown in \cref{fig:finalGoal}.
Thereby one can combine temporal information from intravital microscopy
with well-resolved spatial information from confocal microscopy that can only be obtained after tissue excision.
Here, clusters of fat cells that surround the larger blood vessels were stained after tissue excision and are shown in magenta,
while individual migrating leukocytes were observed during intravital microscopy and are shown in blue.

\subsection*{Acknowledgements}
The authors thank Konrad Buscher, Jian Song, and Lydia Sorokin for discussions relating to the biological motivation and for providing the data of \cref{Fig:MicroscopyImages}.
This work was supported by the Deutsche Forschungsgemeinschaft (DFG), within the Cells-in-Motion Cluster of Excellence (EXC 1003-CiM), University of M\"unster, Germany,
and under Germany's Excellence Strategy EXC 2044 -- 390685587, Mathematics M\"unster: Dynamics-Geometry-Structure.
The research was further supported by the Alfried  Krupp  Prize  for  Young  University  Teachers  awarded  by  the  Alfried  Krupp  von  Bohlen  und Halbach-Stiftung.

\bibliographystyle{plain}
\bibliography{biblio}

\end{document}